\documentclass[english,reqno,11pt]{amsart}
\usepackage{amsmath, amsthm, amsfonts}
\usepackage{hyperref}
\hypersetup{
	colorlinks=true,
		citecolor=blue!60!black,
		linkcolor=red!60!black,
		urlcolor=green!40!black,
		filecolor=yellow!50!black,
	breaklinks=true,
	pdfpagemode=UseNone,
	bookmarksopen=false,
}
\usepackage[many]{tcolorbox}
\usepackage{amsmath,amsthm,amssymb}
\usepackage{amscd,indentfirst,epsfig}
\usepackage{latexsym}
\usepackage{times}
\usepackage{enumerate}
\usepackage{mathrsfs}
\usepackage{stmaryrd}
\usepackage{amsopn}
\usepackage{amsmath}
\usepackage{amssymb,dsfont,mathtools}
\usepackage{amsfonts,bm}
\usepackage{amsbsy,amsmath}
\usepackage{amscd}
\usepackage{xcolor}
\usepackage[mono=false]{libertine}
\usepackage[T1]{fontenc}
\usepackage{amsthm}
\usepackage[normalem]{ulem} 
\usepackage[cal=euler, scr=boondoxo]{}
\usepackage{microtype}

\usepackage{numprint}

\makeatletter
\def \leq {\leqslant}

\def \geq {\geqslant}
\def\e{\varepsilon}
\def \ge {\geq}

\def\R{\mathbb R}
\def\Sb{\mathbb S}

\def\N{\mathbb N}

\def \M {\mathcal{M}}

\def\g{\gamma}
\def \ds {\displaystyle}

\def \d {\mathrm{d}}

\def \Q {\mathcal{Q}}

\def\vet{v_{\ast}}

\newtheorem{theo}{Theorem}
\newtheorem{prop}{Proposition}
\newtheorem{cor}[prop]{Corollary}
\newtheorem{lem}[prop]{Lemma}
\newtheorem{hyp}{Assumptions}
\newtheorem{defi}[hyp]{Definition}
\newtheorem{nb}{Remark}
\newtheorem{exa}[nb]{Example}

\numberwithin{equation}{section}
\def \bfd  {\textbf{(BFD Eq.) }}

\usepackage[many]{tcolorbox}

\newcommand{\B}{\mathcal{B}}
 \newcommand{\HH}{\mathcal{H}}

\renewcommand{\geq}{\geqslant}

\title[Convergence of solutions to Boltzmann-Fermi-Dirac]{Quantitative relaxation towards equilibrium for solutions to the Boltzmann-Fermi-Dirac equation with cutoff hard potentials}

\date{\today}


\author{T. Borsoni}
\address{Sorbonne Université, CNRS, Université Paris Cité, Laboratoire Jacques-Louis Lions (LJLL), F-75005 Paris, France} \email{thomas.borsoni@sorbonne-universite.fr}

\author{B.  Lods}
\address{Universit\`{a} degli
Studi di Torino \& Collegio Carlo Alberto, Department of Economics, Social Sciences, Applied Mathematics and Statistics ``ESOMAS'', Corso Unione Sovietica, 218/bis, 10134 Torino, Italy.}\email{bertrand.lods@unito.it}

\begin{document}

\maketitle

\begin{abstract}
{We provide the first quantitative result of convergence to equilibrium in the context of the spatially homogeneous Boltzmann-Fermi-Dirac equation associated to hard potentials interactions under angular cut-off assumption, providing an explicit -- algebraic -- rate of convergence to Fermi-Dirac steady solutions.} This result complements the quantitative convergence result of~\cite{liulu} and is based upon new  uniform-in-time-and-$\e$ $L^{\infty}$ bound on the solutions.
\end{abstract}

\tableofcontents

\section{Introduction}

\subsection{The Boltzmann-Fermi-Dirac equation}
The scope of the present paper is to derive the \emph{first explicit rate} of convergence to equilibrium for solutions to the spatially homogeneous Boltzmann-Fermi-Dirac equation~\eqref{eq:BFDequation} for hard potentials under the Grad cutoff assumption. The Boltzmann-Fermi-Dirac equation (\bfd  in the rest of the paper) is a modification of the classical Boltzmann equation and describes the interactions between quantum particles {satisfying Pauli's exclusion principle (fermions)}. In the spatially homogeneous setting we are considering here, it takes the form
\begin{equation} \label{eq:BFDequation}
    \partial_t f = \Q^{\varepsilon}_{B}(f,f), \qquad f(0,\cdot) = f^{\rm in},
\end{equation}
where $f \equiv f(t,v) \geq 0$ represents a density of fermions (quantum particles of half-integer spin, e.g. electrons), depending on time $t \geq 0$ and velocity $v \in \R^3$ while the initial datum $f^{\rm in}$ is a given nonnegative distribution density. The Boltzmann-Fermi-Dirac collision operator $\Q^{\varepsilon}_{B}$ is modelling the interactions
between particles and is given, for 
$\varepsilon > 0$ and a suitably integrable $f\geq 0$ such that $1 - \e f \geq 0$, by
\begin{multline} \label{eq:BFDoperator}
   \Q^{\varepsilon}_{B}(f,f)(v) := \iint_{\R^3 \times \Sb^{2}} \Big[f' f'_* (1 - \varepsilon f) (1 - \varepsilon f_*) - f f_* (1 - \varepsilon f') (1 - \varepsilon f'_*) \Big] \times \\
   \times B(v,v_*,\sigma) \, \d  \sigma \, \d  v_*,
\end{multline}
where we used the standard shorthands $f \equiv f(v)$, $f_* \equiv f(v_*)$, $f' \equiv f(v')$, $f'_* \equiv f(v'_*)$ and
\begin{equation} \label{notationvprime}
    v' = \frac{v+v_*}{2} + \frac{|v-v_*|}{2} \, \sigma, \qquad v'_* = \frac{v+v_*}{2} - \frac{|v-v_*|}{2} \, \sigma, \qquad \sigma \in \Sb^{2}.
\end{equation}
The collision kernel $B(v,v_{*},\sigma)$ is providing the  rate at which a given combination of in-going velocities results in a given set of outgoing velocities. In the classical case (for which all quantum effects are neglected), the exact form of the collision kernel can be derived explicitly in the case of interaction driven by inverse power laws repulsive forces and hard-spheres particles~\cite{CCbook}. For  Fermi-Dirac particles, the situation is much more involved and several models co-exist. In particular, for semi-conductor applications, the velocity range is restricted to subset of $\R^{3}$ (periodically repeated Brillouin zones) whereas, in the so-called ``parabolic band'' approximation in semi-conductor, the velocity range is indeed $\R^{3}$. For the sake of simplicity, we choose to present our results in the present form for collision kernels $B$ which assumes the same form for both classical and quantum particles. Considering hard-potential interactions, this corresponds to the choice
 
\begin{subequations}\label{eq:HypBB}
\begin{equation} \label{eqassump:B}
B(v,v_*,\sigma) \equiv b(\cos \theta) \, |v-v_*|^{\gamma}, \qquad \text{ for } \; (v,v_*) \in \R^3 \times \R^3 \; \text{ and } \; \sigma \in \Sb^{2},
\end{equation}
with $\displaystyle \cos \theta = \sigma \cdot \frac{v-v_*}{|v-v_*|}$ and the mapping $b\::\:(-1,1) \to \R^{+}$ is assumed to satisfy  the cutoff assumption
\begin{equation} \label{cutoffassumption}
\|b\|_{1}:=   \|b\|_{L^1(\Sb^2)} :=2\pi \int_{0}^{\pi} b(\cos \theta) \, \sin \theta \, \d  \theta=2\pi\int_{-1}^{1}b(s)\, \d s< \infty.  
\end{equation}
With a slight abuse of notation, in the above~\eqref{cutoffassumption}, we identify the function $b$ (defined over $(-1,1)$) to a function defined over the sphere $\Sb^{2}$ through the identification $\sigma \to \cos\theta$ here above. We consider \emph{hard potential} interactions for which
\begin{equation}\label{eq:gamma}
\gamma \in (0,1].\end{equation}
\end{subequations}
For collision kernel $B$ of the form~\eqref{eqassump:B}, we will use the shorthand notation
$$\Q^{\varepsilon}_{\g,b}(f,f)=\Q^{\varepsilon}_{B}(f,f).$$
We are aware that restricting ourselves to the case of collision kernel $B$ defined by~\eqref{eqassump:B} (valid for both classical and quantum particles) is an important restriction of our analysis and we are confident that most of the results of the present paper can be extend to more physically relevant collision kernels of the form described in~\cite{HeLuPu,WangRen} . We refer to Appendix~\ref{app:COKer} for more details about this. 

As readily seen, the main difference between the classical Boltzmann equation and its quantum counterpart \bfd  lies in the presence of the  \textit{quantum parameter} 
$$\varepsilon:= \frac{(2\pi\hslash)^{3}}{m^{3}\beta} >0$$ 
which depends on the reduced Planck constant $\hslash \approx 1.054\times10^{-34} \mathrm{m}^{2}\mathrm{kg\,s}^{-1}$, the mass $m$ and the statistical weight $\beta$ of the particles species, see \cite[Chapter 17]{CCbook}.  For instance, in the case of electrons $\varepsilon \approx 1.93\times10^{-10}\ll 1$. The parameter $\varepsilon$ quantifies the quantum effects of the model and more precisely ensures Pauli exclusion principle from which solution $f=f(t,v)$ to \bfd  satisfies the \emph{a priori} bound
\begin{equation*}
1-\varepsilon\,f(t,v) \geq 0.\end{equation*}
Formally choosing $\varepsilon = 0$ in~\eqref{eq:BFDoperator} yields the classical Boltzmann operator
\begin{equation} \label{eq:Boltzoperator}
    \Q_B(f,f)=\Q_{B}^{0}(f,f)(v) := \iint_{\R^3 \times \Sb^{2}} \Big[f' f'_* - f f_*\Big] B(v,v_*,\sigma) \, \d  \sigma \, \d  v_*,
\end{equation}
with the same shorthands as in~\eqref{eq:BFDoperator}. In particular, for collision kernel of the form~\eqref{eqassump:B}, we will use for \emph{classical} Boltzmann operator the shorthand notation
$$
\Q_{\gamma,b}(f,f)=\Q^{0}_{\gamma,b}(f,f)=\Q^{0}_{B}(f,f).
$$
We refer the reader to~\cite{lu2001spatially} for results about the Cauchy problem associated to~\eqref{eq:BFDequation} for hard potentials with cutoff, as well as~\cite{CCbook,dolbeaultFD} for more results the Cauchy problem associated with \bfd  \!\!.  We briefly recall in Appendix~\ref{app:cauchy} the results about the Cauchy problem we use in this paper.

\subsection{Notations}  In the following, we define, for $p \geq 1$ and $k \in \R$, the Lebesgue space {$L^{p}_{k}=L^{p}_{k}(\R^{3})$} through the norm
\begin{equation} \label{eqdef:Lpsnorm}\displaystyle \|f\|_{L^p_{k}} := \left(\int_{\R^{3}} \big|f(v)\big|^p \, 
\langle v\rangle^{kp} \, \d v\right)^{\frac{1}{p}}, \qquad L^{p}_{k}(\R^{3}) :=\Big\{f\::\R^{3} \to \R\;;\,\|f\|_{L^{p}_{k}} < \infty\Big\}\,,\end{equation}
where $\langle v\rangle :=\sqrt{1+|v|^{2}}$, $v\in \R^{3}.$ For $k=0$, we simply denote $\|\cdot\|_{p}$ the $L^{p}$-norm. For $p=\infty$, $\|\cdot\|_{\infty}$ will denote the usual  essential supremum of a given measurable function. We also define, for $k \in \R$, the Orlicz space $L^{1}_{k}\log L(\R^{3})$ as
$$ L^{1}_{k}\log L(\R^{3}) =\Big\{f\::\R^{3} \to \R\;;\,\int_{\R^{3}}\langle v\rangle^{k} \, |f(v)| \, \log^+ |f(v)| \, \d v \, < \infty\Big\},$$
as well as the quantity
$$
\|f\|_{L^1_k \log L} := \int_{\R^{3}}\langle v\rangle^{k} \, |f(v)| \, |\log |f(v)| | \, \d v,
$$
which we highlight is not a norm on the Orlicz space, but is finite as soon as there is an $s > 0$ such that $f \in L^1_{k+s}(\R^3) \cap L^1_k \log L(\R^3)$. As usual, for $k=0$, we simply write $\|\cdot\|_{L\log L}=\|\cdot\|_{L^{1}_{0}\log L}$ and $L\log L(\R^{3})=L^{1}_{0}\log L(\R^{3}).$

\subsection{Relaxation to equilibrium} In contrast to what occurs for classical gases, quantum gases of fermions exhibit \emph{two} distinct families of steady states. First, $\Q_{B}^{\varepsilon}(\M_{\varepsilon})=0$ when $\M_{\varepsilon}$ is the following Fermi-Dirac statistics:  
  \begin{defi}[\textit{\textbf{Fermi-Dirac statistics}}]\label{defi:FDstats} Given $\varrho >0, u\in\R^3, E >0$ and $\e >0$ satisfying 
    \begin{equation}\label{eq:Intheta}
 5\,E>\left(\frac{3\varepsilon\varrho}{4\pi}\right)^{\frac23},\end{equation}
we denote by $\M_{\varepsilon}$ the \emph{unique} Fermi-Dirac statistics  
\begin{equation}\label{eq:FDS}
\M_{\varepsilon}(v)=\frac{\exp(a_{\varepsilon} + b_{\varepsilon}|v-u|^{2})}{1+\varepsilon\,\exp(a_{\varepsilon} + b_{\varepsilon}|v-u|^{2})}=: \frac{M_{\varepsilon}(v)}{1+\varepsilon\,M_{\varepsilon}(v)},
\end{equation}
with $a_{\varepsilon} \in \R$ and $b_{\varepsilon} < 0$ defined in such a way that
$$\int_{\R^{3}}\M_{\varepsilon}(v)\left(
\begin{array}{c}1\\v \\|v-u|^{2}
\end{array}\right) \, \d v
=\left(\begin{array}{c}\varrho \\\varrho\,u \\3\varrho\,E\end{array}\right)\,.$$ 
\end{defi}
The existence and uniqueness of Fermi-Dirac statistics satisfying~\eqref{eq:Intheta} has been established in \cite[Proposition 3]{lu2001spatially}.   Note that  $M_{\varepsilon}$ is here a suitable Maxwellian distribution that allows to recover in the classical limit $\varepsilon \to 0$ the Maxwellian equilibrium.

\noindent
Besides the Fermi-Dirac statistics~\eqref{eq:FDS}, the distribution
\begin{equation}\label{eq:dege}
F_{\varepsilon}(v)=\begin{cases}
\quad \varepsilon^{-1}  \qquad &\text{ if } \quad |{v-u}|\leq \left(\dfrac{3\varrho\,\varepsilon}{4\pi} \right)^{\frac{1}{3}}, \\
\quad 0 \qquad &\text{ if } \quad |{v-u}|> \left(\dfrac{3\varrho\,\varepsilon}{4\pi}\right)^{\frac{1}{3}}\,,\end{cases}
\end{equation}   
can also  be a stationary state, as $\Q_{B}^{\varepsilon}(F_{\varepsilon})=0$.  Such a degenerate state, referred to as a \emph{saturated Fermi-Dirac} stationary state, can occur for very cold gases (with an explicit condition on the gas temperature). More precisely, for any $\e >0,$ one can define the \emph{Fermi temperature} associated to $\varrho >0,$ $\e> 0,$ as
$$\bm{T}(\varrho,\e):=\frac{1}{2}\left(\frac{3\e\varrho}{4\pi}\right)^{\frac{2}{3}}.$$
Then, given 
$0 \leq f \in L^{1}_2(\R^3) \setminus \{0\}$ with 
\begin{equation}\label{eq:fEf}\int_{\R^3} f(v) \begin{pmatrix} 1 \\ v \\ |v|^2  \end{pmatrix} \, \d  v  =  \begin{pmatrix} \varrho \\ \varrho u  \\ 3 \varrho E + \varrho |u|^2 \end{pmatrix}\end{equation}
and $\e >0$, the ratio $r_E$  between the actual temperature of $f$ and its Fermi temperature defined as
$$r_E:=\frac{E}{\bm{T}(\varrho,\e)}$$
is an a-dimensional number which plays a crucial role in the dynamic of \bfd \!\!. Indeed, it has been shown in~\cite{lu2001spatially} that, given $\e >0$, the following holds
$$1-\e f \geq 0 \implies r_E \geq \frac{2}{5}$$
with moreover the following dichotomy:
\begin{enumerate}
    \item $1-\e f\geq 0$ and $r_E=\frac{2}{5}$ if and only if $f=F_\e$ as defined in~\eqref{eq:dege};
    \item $1-\e f \geq0$ and $r_E > \frac{2}{5}$ if and only if there exists a unique Fermi-Dirac statistics $\M_\e=\M_\e^f$ with same mass, momentum and energy that $f$.
\end{enumerate}
Observe here that
\begin{equation}\label{eq:sat}
r_E=\frac{2}{5} \iff \e=\e_{\text{sat}}=\frac{4\pi \,(5\,E)^{\frac{3}{2}}}{3\varrho}\,
\end{equation}  whereas  the inequality~\eqref{eq:Intheta} exactly means that $\e \in (0,\e_{\mathrm{sat}}).$ In all the sequel, for  given $\varrho,E >0$, we will always implicitly assume that $\e \in (0,\e_{\mathrm{sat}}).$

As we will see in the next section, the fact that an initial distribution close to such degenerate state makes $1-\varepsilon f$ arbitrarily small in non negligible sets affects drastically the speed of convergence and one of the crucial points of our analysis will be to show that, for suitable initial datum $f^{\rm in}$,  there exist $c \in (0,1)$ and $\kappa_{0}$ (depending on $c$) such that solutions $f^\varepsilon$ to \bfd  satisfy
$$1-\varepsilon f^{\varepsilon}(t,v) \geq \kappa_0, \qquad \forall \varepsilon \in (0,c\varepsilon_{\text{sat}}).$$

\emph{In all the sequel, we will always consider solutions to~\eqref{eq:BFDequation} associated to the operator $\Q_B^{\varepsilon}=\Q^{\varepsilon}_{\gamma,b}$, i.e. considering kernels $B$ of the form~\eqref{eq:HypBB}. 
We also always consider nonnegative initial datum $f^{\rm in} \in L^1_2(\R^3)$ 
and (conservative) solutions $f^{\varepsilon}$ to \bfd  associated to $f^{\rm in}$  whose existence and uniqueness are recalled in Appendix~\ref{app:cauchy}. In particular, 
$$\int_{\R^{3}}f^{\varepsilon}(t,v)\left(
\begin{array}{c}1\\v \\|v-u^{\rm in}|^{2}
\end{array}\right) \, \d v
=\int_{\R^{3}}f^{\rm in}(v)\left(
\begin{array}{c}1\\v \\|v-u^{\rm in}|^{2}
\end{array}\right) \, \d v
=\left(\begin{array}{c}\varrho^{\rm in} \\\varrho^{\rm in}\,u^{\rm in} \\3\varrho^{\rm in}\,E^{\rm in}\end{array}\right)\,\qquad \forall t \geq0$$
with $\varrho^{\rm in} >0$, $u^{\rm in} \in \R^3$ and $E^{\rm in} >0$.
We will also, unless otherwise stated, consider the associated Fermi-Dirac statistics $\M_{\varepsilon}=\M_\varepsilon^{f^{\rm in}}$ sharing the same mass, momentum and kinetic energy as $f^{\rm in}$, i.e.
$$\int_{\R^{3}}\M_{\varepsilon}^{f^{\rm in}}(v)\left(
\begin{array}{c}1\\v \\|v-u^{\rm in}|^{2}
\end{array}\right) \, \d v
=\left(\begin{array}{c}\varrho^{\rm in} \\\varrho^{\rm in}\,u^{\rm in} \\3\varrho^{\rm in}\,E^{\rm in}\end{array}\right)\,.$$}
For such solutions, non quantitative results about the long-time behaviour of solutions to \bfd  have been already obtained in the literature. In particular, we have the following recent result from~\cite{liulu}:
\begin{theo}[\textit{\textbf{Liu \& Lu (2023)}}] Assume that the collision kernel $B=B(v,\vet,\sigma)$ satisfies~\eqref{eq:HypBB} with moreover 
\begin{equation}\label{eq:bComplPos}
b(\cos\theta) \geq \sum_{n=0}^\infty a_n \cos^{2n}(\theta), \qquad \theta \in (0,\pi), \qquad a_n \geq0 \qquad \forall n \in \N.\end{equation}
For any initial datum $f^{\rm in} \in L^1_2(\R^3)$ with $0 \leq f^{\rm in} \leq \varepsilon^{-1}$, the unique conservative (mild) solution $f^{\varepsilon}=f^{\varepsilon}(t,v)$ to \bfd  with initial datum $f^{\rm in}$ is such that  
$$\lim_{t \to \infty}\left\|f^{\varepsilon}(t)-\M_\varepsilon\right\|_{L^1_2}=0$$
where $\M_\varepsilon$ is the unique Fermi-Dirac statistics with same mass, momentum and kinetic energy as $f^{\rm in}$.
\end{theo}
\begin{nb}
The additional assumption~\eqref{eq:bComplPos} on the angular kernel $b=b(\cos\theta)$ in the above theorem means that, as a function of $\cos^2(\theta)$, $b(\cos\theta)$ is \emph{completely positive}. It is of course satisfied for instance if $b(\cos\theta)$ is bounded by below away from zero (corresponding to $a_0=\inf_\theta b(\cos\theta)$ and $a_n=0$ for $n \geq 1$) which is the simplified setting in which our main result (see Theorem~\ref{theo:main}) will hold true.
\end{nb}
 As mentioned earlier, up to our knowledge, no quantitative estimates for the rate of convergence towards $\M_\varepsilon$ exist in the literature and it is the purpose of our work to fill this blank.

 \subsection{The role of relative entropy} As very well documented, the main tool to provide \emph{quantitative} rate of relaxation to equilibrium is related to entropy/entropy production. For $\varepsilon>0,$ we introduce the Fermi-Dirac entropy: 

\begin{equation}\label{eq:FDentro}
\HH_{\varepsilon}(f)=  \int_{\R^3} \left[ f\log f+\e^{-1}(1-\varepsilon f)\log (1-\varepsilon f)\right] \, \d v,
\end{equation}
well-defined for any $0 \leq f\leq \varepsilon^{-1}.$ One can then show that $\HH_{\varepsilon}(f)$ is a Lyapunov function for~\eqref{eq:BFDequation}, i.e.
$$\dfrac{\d}{\d t}\HH_{\varepsilon}(f(t))=:-\mathscr{D}_{\varepsilon}(f(t)) \leq 0$$
for any suitable solution to~\eqref{eq:BFDequation} where the entropy production is defined, assuming $1-\e f > 0$ almost everywhere, by
\begin{multline}\label{eq:product}
\mathscr{D}_{\varepsilon}(f):=\frac{1}{4}\int_{\R^{3}\times\R^{3} \times \Sb^{2}}\left[\varphi_{\varepsilon}(f')\varphi_{\varepsilon}(f'_{*})-\varphi_{\varepsilon}(f)\varphi_{\varepsilon}(f_{*})\right] \log\left(\dfrac{\varphi_{\varepsilon}(f')\varphi_{\varepsilon}(f'_{*})}{\varphi_{\varepsilon}(f')\varphi_{\varepsilon}(f'_{*})}\right) \times\\
\times (1 - \e f) (1 - \e f_*) (1 - \e f') (1 - \e f'_*) \,  B(v,\vet,\sigma) \, \d v \, \d \vet \, \d\sigma,
\end{multline}
where 
\begin{equation}\label{eq:varphiE}
\varphi_{\varepsilon}(x)=\frac{x}{1-\varepsilon x}, \qquad x \in [0,\varepsilon^{-1}).\end{equation}
In particular, introducing the \emph{relative entropy} 
$$\mathcal{H}_{\varepsilon}\left(f\,\big|\M_{\varepsilon}\right)=\HH_{\varepsilon}(f)-\HH_{\varepsilon}(\M_{\varepsilon})$$
which is nonnegative if $f$ and $\M_{\varepsilon}$ share the same mass, momentum and kinetic energy, 
one has, for any suitable solution to \bfd \!\!, 
\begin{equation}\label{eq:Hep}
\dfrac{\d}{\d t}\mathcal{H}_{\varepsilon}\left(f(t)\big|\M_{\varepsilon}\right)=-\mathscr{D}_{\varepsilon}(f(t)) \leq 0.\end{equation}
According to Csisz\'ar-Kullback-Pinsker inequality (see Appendix~\ref{app:cauchy}), the relative entropy controls the  $L^{1}_{k}$-norm of the difference $f(t)-\M_{\varepsilon}$ and this is what makes entropy/entropy production estimate a powerful tool for proving the convergence towards equilibrium. Indeed, if one is able to prove a \emph{functional inequality} of the form
$$\mathscr{D}_{\varepsilon}(f) \geq \bm{\Theta}\left(\mathcal{H}_{\varepsilon}\left(f \big| \M_{\varepsilon}\right)\right), \qquad \forall f \in \mathcal{C}$$
where $\mathcal{C}$ is a suitable class of functions $0 \leq f \in L^{1}_{2}(\R^{3})$ with $0 \leq f\leq \varepsilon^{-1}$ and a \emph{superlinear} mapping $\bm{\Theta}\::\:\R^{+} \to \R^{+}$, Eq.~\eqref{eq:Hep} implies
$$\dfrac{\d}{\d t} \mathcal{H}_{\varepsilon}\left(f(t)\big|\M_{\varepsilon}\right) \leq -\bm{\Theta}\left(\mathcal{H}_{\varepsilon}\left(f \big| \M_{\varepsilon}\right)\right)$$
provided solutions to \bfd  belong to the class $\mathcal{C}$. Then, a Gr\"onwall-type argument provides a rate of convergence to equilibrium of the relative entropy. For instance, if the entropy production controls the relative entropy in a linear way, corresponding to $\bm{\Theta}(u)=\lambda u$, for some $\lambda >0$ and any $u \geq 0$, then one would get the exponential relaxation to equilibrium
$$\mathcal{H}_{\varepsilon}\left(f(t)\big|\M_{\varepsilon}\right) \leq \exp\left(-\lambda t\right)\mathcal{H}_{\varepsilon}\left(f(0)\big|\M_{\varepsilon}\right), \qquad \forall t \geq0.$$
provided one is able to show that solutions to \bfd  belong to the class $\mathcal{C}$. Such a decay, combined with Csisz\'ar-Kullback-Pinsker inequality, yield an explicit rate of convergence of $f(t)$ towards $\M_\e$ in $L^1_k$ (or even $L^p_k$, $p > 1$) topology. Other kinds of functional $\bm{\Theta}$ can of course be considered and such a strategy has been efficiently applied to the study of the long-time behavior for classical  gases, corresponding to $\varepsilon=0$, for which suitable functional inequalities linking the Boltzmann relative entropy
$$\mathcal{H}_{0}\left(f\big|\M_{0}\right)=\int_{\R^{3}}f\log f\d v-\int_{\R^{3}}\M_{0}\log \M_{0}\d v$$
and the entropy production
$$\mathscr{D}_{0}(f):=\frac{1}{4}\int_{\R^{3}\times\R^{3} \times \Sb^{2}}\left[f'\,f'_{*}- f\,f_{*}\right]\log\left(\dfrac{ f' f'_{*}}{f\,f_{*}}\right)B(v,\vet,\sigma)\d v\d \vet\d\sigma,$$
have been obtained, starting with the pioneering works~\cite{cercioriginal,CarlenCar,ToVi} and culminating with a celebrated result in~\cite{villani}. An improvement of the result of~\cite{villani} (reducing the required regularity of $f$, up to the mere $L^p$ or even $L\log L$ estimate) has been derived in~\cite{alonso2017} and can be formulated as
\begin{theo}\label{theo:Vil} Assume the existence of $b_{0} > 0$ and $\beta_{\pm} \geq 0$ such that
$$B(v,\vet,\sigma) \geq b_{0}\min\left(|v-\vet|^{\beta_{+}},\,|v-\vet|^{-\beta_{-}}\right), \qquad B(v,\vet,\sigma)=B(|v-\vet|,\cos\theta)\,.$$
Given $K_{0} > 0,\,A_{0}>0$, $ {q_{0}\geq 2,}$  we consider the class of functions
$$
\mathcal{C}_{0}=\Bigg\{g \in L^{1}_{2}(\R^{3})\,\text{  such that } \,g(v) \geq K_{0}\exp\left(-A_{0}|v|^{ {q_{0}}}\right),\,\;\;\forall v \in \R^{3}\Bigg\}\,.$$ Then, given $1 < p < \infty$ and $\delta >0$, defining $s=2+\frac{2+\beta_-}{\delta}$, for any 
$$g \in L^p(\R^3) \cap L^1_s\log L(\R^3) \cap \mathcal{C}_0 \cap L^1_{s+q_0}(\R^3)$$ it holds
\begin{equation}\label{eq:D0p}
\mathscr{D}_0(g) \geq A_{\delta,p}(g )\mathcal{H}_0\left(g \big|\M_0^{g}\right)^{\alpha}
\end{equation}
    where $\alpha=(1+\delta)(1+ \frac{p\beta_+}{3(p-1)})$ and $A_{\delta,p}(f)$ depend on the parameters $K_0,q_0,A_0,\beta_\pm,p$, upper and lower bounds to $\|g\|_1$ and $\|g\|_{L^1_2}$, as well as upper bounds for $\|g\|_{L^1_{s}\log L(\R^3)},\|g\|_{p},{\|g\|_{L^1_{s+q_0}}}$.  
\end{theo}
\begin{nb} Notice that the above assumption on $B$ is satisfied for collision kernel of the form~\eqref{eq:HypBB} with $\beta_+=\gamma,\beta_-=0$ provided that 
$b(\cos \theta) \geq b_0.$ We wish also to point out that we will apply the above result to
$$g=\varphi_\e(f(t))=\frac{f(t)}{1-\e f(t)}, \qquad f(t) \text{ solution to~\eqref{eq:BFDequation}},$$
and we can check easily that, in such a case, if $1-\e f(t) \geq \kappa_0$, then $A_{\delta,p}(\varphi_\e(f(t))$ can be be bounded away from zero uniformly with respect to time (and $\e$). See the proof of Theorem~\ref{theo:main}. 
\end{nb}

Such a result has been recently adapted by the first author in~\cite{borsoni2023extending} to the case of the Fermi-Dirac entropy thanks to a suitable link between the Boltzmann relative entropy $\mathcal{H}_{0}(\varphi_{\varepsilon}(f)|\M_{\varphi_{\varepsilon}(f)})$ and the Fermi-Dirac relative entropy $\mathcal{H}_{\varepsilon}(f|\M_\varepsilon)$. Typically,
\begin{prop}\label{prop:TB} Given $\kappa_0 \in (0,1)$, there is a positive constant $C(\kappa_0) >0$, that could be made explicit, such that, for any $\varepsilon > 0$ and nonnegative $f \in L^{1}_{2}(\R^3) \setminus \{0\}$ such that 
$$1-\varepsilon f \geq \kappa_{0},$$
it holds
$$\mathcal{H}_{\varepsilon}(f\big|\M_{\varepsilon}^f) \leq \mathcal{H}_{0}\left(\varphi_{\varepsilon}(f)\big| \M_0^{\varphi_{\varepsilon}(f)}\right) \leq C(\kappa_{0})\mathcal{H}_{\varepsilon}(f\big|\M_{\varepsilon}^f) $$ and
$$\kappa_0^4 \, \mathscr{D}_0\left(\varphi_{\varepsilon}(f)\right)\leq \mathscr{D}_{\varepsilon}(f) \leq \mathscr{D}_0\left(\varphi_{\varepsilon}(f)\right),$$
where we recall that we defined $\varphi_{\e}(x) = \frac{x}{1-\e x}$ for $x \in [0,\e^{-1})$ in \eqref{eq:varphiE}.
\end{prop}

The comparison provided in Proposition~\ref{prop:TB} is the key point of our analysis as it allows to deduce suitable entropy/entropy production estimates in the Fermi-Dirac case. To that purpose, for any $\varepsilon >0,$ and $\kappa_0>0,$ $K_{0} > 0,\,A_{0}>0$, ${q_{0}\geq 2},$ $E >0,\varrho >0,$ $ u\in \R^3$, we introduce the class
\begin{multline}\label{eq:classCe}
\mathcal{C}_\varepsilon=\Bigg\{f \in L^{1}_{2}(\R^{3}) \,\text{ satisfying~\eqref{eq:fEf} and such that }\\
f(v) \geq K_{0}\exp\left(-A_{0}|v|^{ {q_{0}}}\right)\,\;\,\text{ and } 1-\varepsilon\,f(v) \geq \kappa_{0}, \quad  \forall v \in \R^{3} 
\Bigg\}\,.\end{multline} 
One can  easily deduce from Theorem~\ref{theo:Vil} that, for any $\varepsilon,\delta >0$ and $1<p<\infty$, an estimate of the type
$$\mathscr{D}_{\varepsilon}(f) \geq \tilde{A}_{\delta,p}(f)\,\mathcal{H}_\varepsilon\left(f\,\big|\M_\varepsilon\right)^\alpha \qquad \forall f \in \mathcal{C}_\varepsilon \cap L^p(\R^3) \cap L^1_s\log L(\R^3)$$
where $\tilde{A}_{\delta,p}(f)$ depends explicitly on $\kappa_0,K_0,A_0, {q_0},\varrho,E$ and upper bounds on $\|f\|_{L^1_s\log L},\|f\|_p$ with $s$ defined in Theorem~\ref{theo:Vil}. 
As in the classical case, the key point to apply such a functional inequality to solutions to \bfd  is then to determine suitable conditions on the initial distribution $f^{\rm in}$ ensuring that the unique solution $f^{\varepsilon}=f^{\varepsilon}(t,v)$ to~\eqref{eq:BFDequation} satisfies
$$f^\varepsilon(t,\cdot) \in \mathcal{C}_\varepsilon \cap L^p(\R^3) \cap L^1_s\log L(\R^3) \qquad \forall t \geq 0.$$
 The main technical difficulty, which explains as already mentioned why quantitative rate of convergence for \bfd  has not be obtained yet, is of course to create or propagate the lower bound
\begin{equation}\label{eq:low}
1-\varepsilon f^{\varepsilon}(t,v) \geq \kappa_0 \qquad \forall t >0,\quad v \in \R^3.\end{equation}

In the study of the Landau-Fermi-Dirac equation, such pointwise lower bound have been obtained by improving the mere $L^\infty$ bounds thanks to the regularisation mechanism induced by the diffusive nature of the collision operator. This allows to prove the \emph{appearance} of some $L^\infty$ bound, \emph{independent of $\varepsilon$}  for the solution to the Landau-Fermi-Dirac equation and, up to reducing slightly  the parameter $\varepsilon$, to obtain the lower bound~\eqref{eq:low}.

When trying to adapt such a strategy to \bfd \!\!, the major difficulty  lies in the fact that, for \emph{cut-off} interactions (see~\eqref{cutoffassumption}), no smoothing effect of the solution is expected  and a new route has to be followed to deduce the non saturation estimate.\footnote{
We mention that a result of uniform-in-time $L^\infty$ bound was already proposed in~\cite{WangRen} (with $\e = 1$ and different kinds of collision kernels), however it seems to us that counter-examples to that result can be constructed. In particular,  choosing $a=b,\beta=0$ in ~\cite[Theorem 1.7]{WangRen},  one could consider an initial distribution $f_0$ with $L^\infty$ norm equal to $1$, and \cite[Theorem 1.7]{WangRen} would imply that solutions $f(t)$ to \bfd  are such that $\|f(t)\|_{\infty} \leq \frac{1}{6}$ for all time $t \ge 0$.} We insist on the fact that, as already observed in~\cite{luwennberg,liulu}, such non saturation estimate is the crucial argument to provide quantitative rate of convergence to equilibrium for quantum Boltzmann equation.

\color{black}

We now describe with more details our main results in the next subsection.

\subsection{Main results} Recall that we are dealing with solutions to \bfd  as constructed in Theorem~\ref{theo:cauchy} in Appendix~\ref{app:cauchy}. As mentioned earlier, one of the main aspects of our strategy is to derive $L^\infty$ bounds for solutions to \bfd  which are independent of $\varepsilon$. Besides its fundamental role for the long-time behaviour of solutions to \bfd \!\!, such a result has its own independent interest and is one of the main results of our contribution 
\begin{theo}[\textbf{\textit{Uniform-in-$\e$ $L^{\infty}$ bound}}] \label{theorem:linftybound}
Let $\gamma \in (0,1]$ and an angular kernel $b$ satisfying the cutoff assumption~\eqref{cutoffassumption}. Let $0 \leq f^{\rm in} \in L^1_3(\R^3) \cap L^{\infty}(\R^3)$. Then there exists an explicit $\mathbf{C}_{\infty} > 0$, depending only on $\gamma$, $b$ and $f^{\rm in}$ only through $\varrho^{\rm in}$, $E^{\rm in}$ and upper-bounds on its $L^\infty$ and $L^1_3$ norms, such that for any $\e \in (0, \|f^{\rm in}\|_{\infty}^{-1})$, the unique solution $f^\e$ to~\eqref{eq:BFDequation} associated to~$\e$, the collision kernel defined by~\eqref{eq:HypBB}  and initial datum $f^{\rm in}$ satisfies 
\begin{equation} \label{eqtheorem:linftybound}
    \sup_{t \geq 0} \|f^{\e}(t)\|_{\infty} \leq \mathbf{C}_{\infty}.
\end{equation}
\end{theo}

\begin{nb}
The assumption of $f^{\rm in}$ belonging to $L^1_3$ in Theorem~\ref{theorem:linftybound} could be actually recast into $f^{\rm in}$ belonging to $L^1_{s}$ for some arbitrary $s$ such that $s > 2$ and $s \geq 3 \gamma$. In this case, $\mathbf{C}_{\infty}$ would also depend on $s$.
\end{nb}

\noindent The peculiarity of the result presented in Theorem~\ref{theorem:linftybound} lies in the fact that the bound is \emph{independent of $\e$}. This is one of the main points of the strategy we adopt here, reminiscent from a similar one in~\cite{ABL} (see also~\cite{ABDLsoft}). The clear advantage of such a bound is that, for a given \emph{fixed} $f^{\rm in} \in L^\infty(\R^3)$, one can always choose a small enough $\e$ such that the non-saturation condition~\eqref{eq:low} holds for any time $t \geq0$:
\begin{cor} \label{corlinfkappa0} Consider the assumptions of Theorem ~\ref{theorem:linftybound} along with the same notations. Then for any
 $\kappa_0 \in (0,1)$ and $\e \in \left(0, (1-\kappa_0) \mathbf{C}_{\infty}^{-1} \right]$, the mentioned solution $f^\e$ to~\eqref{eq:BFDequation} is such that
\begin{equation} \label{eq:corlinfkappa0}
  \qquad \qquad \qquad \qquad \qquad  1 - \e f^{\e}(t,v) \geq \kappa_0, \qquad \forall \, (t,v) \in \R_+ \times \R^3.
\end{equation}
\end{cor}This non-saturation property allows to greatly simplify various studies on the equation: in this event, and especially in the cutoff case (that is, when~\eqref{cutoffassumption} holds), very large parts of the study of solutions to the Boltzmann-Fermi-Dirac equation can be recast into the study of the classical Boltzmann equation thanks to a suitable comparison argument (see Proposition~\ref{prop:TB}). To name two examples particularly important for our analysis, this allows to transfer the entropy inequalities from the classical to the fermionic case thanks to Proposition~\ref{prop:TB} and also to deduce in a straightforward way a Maxwellian lower bound on the solutions to \bfd  by simply resuming the proof of~\cite{pulvirentimaxellian} valid in the classical case (see Subsection~\ref{subsect:maxwellianlb}). 

We insist here on the fact that, our approach consists in choosing \emph{first} an initial datum $f^{\rm in}$ and subsequently pick $\e^{\rm in} >0$ small enough for the rest of our analysis to apply uniformly with respect to $\e \in (0,\e^{\rm in}]$. It seems possible to adopt a similar viewpoint by choosing first $\e >0$ and then determine the class of all initial data for which the results of the paper do hold. We did not pursue this line of investigation.
\bigskip

\noindent With this at hands,  we can easily deduce the main result for the long-time behaviour of the solutions, where we recall that we consider solutions to~\eqref{eq:BFDequation} as constructed in~\cite{lu2001spatially} (see Appendix~\ref{app:cauchy}):
\begin{theo}[\textbf{\textit{Explicit rate of convergence to equilibrium}}]\label{theo:main}
Let $B$ be a collision kernel of the form~\eqref{eq:HypBB} with $\gamma \in (0,1]$ and an angular kernel $b$ satisfying the cutoff assumption~\eqref{cutoffassumption} and such that $b \geq b_0$ for some $b_0 > 0$. Let $0 \leq f^{\rm in} \in L^1_3(\R^3) \cap L^{\infty}(\R^3)$. Then there exist some explicit $\mathbf{C}_{\HH} > 0$ and $\e^{\rm in} > 0$, depending only on $\gamma$, $b$, $\varrho^{\rm in}$, $u^{\rm in}$, $E^{\rm in}$, $\|f^{\rm in}\|_{L^1_3}$ and $\|f^{\rm in}\|_\infty$, such that for any $\e \in (0, \e^{\rm in}]$, the unique solution $f^\e$ to~\eqref{eq:BFDequation} associated to~$\e$ and the initial datum $f^{\rm in}$ satisfies, for all $t \geq 0$,
\begin{equation} \label{eqtheorem:explicitcvH}
   \mathcal{H}_{\varepsilon}(f^{\e}(t)\,\big|\M_\e^{f^{\rm in}})  \leq \mathbf{C}_{\HH} \, (1 + t)^{-\frac{1}{\gamma}}.
\end{equation}
In particular, for any $p \geq 1$ and $k \geq 0$, there exist $\mathbf{C}_{\HH,p,k} > 0$, with the same properties as $\mathbf{C}_{\HH}$, such that for all $t \geq 0$,
\begin{equation} \label{eqtheorem:explicitcvLpk}
   \left\|f^{\e}(t) - \M_\e^{f^{\rm in}} \right\|_{L^p_k}  \leq \mathbf{C}_{\HH,p,k} \, (1 + t)^{-\frac{1}{2 p \gamma}}.
\end{equation}
\end{theo}
\begin{nb} Notice that \eqref{eqtheorem:explicitcvH} holds true with the choice 
$$\e^{\rm in}=(1-\kappa_0) \, \mathbf{C}_{\infty}^{-1}$$
and the same choice applies to \eqref{eqtheorem:explicitcvLpk} in the case $p=1, k=0$. In the case $k >0$ or $p >1$, additional smallness restriction is required on the parameter $\e$ since an additional control of $\|\M_\e^{f^{\rm in}}\|_{\infty}$ or $\|\M_\e^{f^{\rm in}}\|_{L^1_{2k}}$ is actually needed.  For the same reason, the constants actually depend on $u^{\rm in}$ only in the case of the $L^p_k$ norm with $k > 0$. We refer to Section~\ref{sec:maintheo} for more details.
\end{nb}
\begin{nb} We point out right away that the rate of convergence to equilibrium established in Theorem~\ref{theo:main} is clearly not the optimal rate of convergence. As well-known for kinetic equations associated with hard-potentials (recall here $\gamma \in (0,1]$), such a rate of convergence can be upgraded into an exponential relaxation governed by the spectral gap of the linearized collision operator. To implement such an upgrade, one would need to adopt the following strategy 
\begin{enumerate}[a)]
    \item Show a quantitative exponential rate of convergence to $\M_\e$ for \emph{close-to-equilibrium} initial state. This means that one can construct $\lambda >0$ and an explicit $\delta >0$ such that
$$\|f^{\rm in}-\M_\e^{f^{\rm in}}\|_{L^1_k} \leq \delta \implies \|f^{\varepsilon}(t)-\M_\e^{f^{\rm in}}\|_{L^1_k} \leq C\exp\left(-\lambda t\right)$$
for some $k >0$ large enough and some explicit $C >0$ (depending on $f^{\rm in}$).
\item Second, using Theorem~\ref{theo:main}, one can \emph{explicit} the time  $T >0$ needed for general solutions $f^{\varepsilon}$ to enter the neighbourhood of $\M_\e^{f^{\rm in}}$ determined by $\delta$. Then, restarting the evolution from time $T >0$, we get that the rate of convergence to equilibrium is exponential.\end{enumerate}
A full analysis of the close-to-equilibrium solutions to \bfd    (i.e. full proof of the above point (a)) is still missing and is based upon  a careful spectral analysis of the linearized operator associated to $\Q^{\varepsilon}_B$ around $\M_\e^{f^{\rm in}}$. We are confident that it can be implemented, deriving first a spectral gap estimate in a Hilbert setting and extending then this spectral analysis to the $L^1_k$ by factorization and enlargement techniques following the line of~\cite{GMM,Mouhot}. Notice that such an approach has been successfully applied in the study of the Landau-Fermi-Dirac equation in~\cite{ABL}.
\end{nb}

As mentioned already, such a Theorem is, to our knowledge, the first result providing an explicit and quantitative rate of convergence to equilibrium for solutions to \bfd  whereas several qualitative results were available in the literature.  As said before, the crucial ideas of our strategy are the following:
\begin{itemize}
    \item First, thanks to the $L^\infty$-bound independent of the quantum parameter $\varepsilon,$ one can, for a given initial datum $f^{\rm in}$, determine a whole family of $\varepsilon$ for which the non-saturation estimate~\eqref{eq:low} holds true uniformly in time.
    \item Second, taking profit of this non saturation condition, we can exploit several new and insightful comparison arguments -- specially tailored for our analysis and with their own independent interest -- which allow to deduce results regarding \bfd  from the analogue ones already obtained in the classical case as exposed in the comprehensive survey~\cite{alonso2022boltzmann} (see also~\cite{alonso2017}).
\end{itemize}
Not only such an approach allows, in some sense, to treat in a same formalism the quantum and classical Boltzmann equations, but it also appears quite straightforward and elegant, fully exploiting the tools developed in~\cite{alonso2022boltzmann} for the study of kinetic equations with hard potentials and cut-off assumptions. As mentioned in the last remark, we decided to focus on the main mathematical challenge for the rate of convergence, which is the control of the relative entropy, discarding in our analysis the optimality of the rate of convergence which should be easy to deduce from our result thanks to the recent methods developed in~\cite{GMM}.

\subsection{Organization of the paper} The paper is organized as follows: after this Introduction, we present in Section~\ref{sec:LpLinf} the full proof of the main estimates for solutions to \eqref{eq:BFDequation} culminating with the proof of Theorem~\ref{theorem:linftybound}. It is obtained from the representation of solutions to \bfd  as suitable ``sub-solutions'' (in the sense of~\eqref{eq:fsubsolQ0}) to an equation very similar to the classical Boltzmann equation. We first describe such a representation in Subsection~\ref{sec:sub}, then deduce suitable Young's type estimates for the collision operator associated to such an equation, in the spirit of the fundamental results in~\cite{alonso2010estimates,alocargam} (following the more recent exposition in~\cite{alonso2017,alonso2022boltzmann}). In particular, $L^2_\gamma$ bounds for solutions to \bfd  are deduced, uniformly with respect to $\e >0$ and time, in Proposition~\ref{prop:L2bound} and, as in~\cite{alonso2022boltzmann}, this allows us to derive the full proof of Theorem~\ref{theorem:linftybound}. We then provide, in Section~\ref{sec:maintheo}, the complete proof of our main convergence result, Theorem~\ref{theo:main}. It is deduced from the results of Section~\ref{sec:LpLinf} combined with suitable pointwise lower bounds, well-known for the classical Boltzmann equation~\cite{pulvirentimaxellian} and easily adapted to \bfd  in Proposition~\ref{prop:maxwellianlb}. In Appendix~\ref{app:COKer}, we present examples of physically relevant collision kernels for quantum kinetic equations and describe how the results of the paper could be adapted to such general collision models. In Appendix~\ref{app:cauchy}, we briefly recall results regarding existence and uniqueness of solutions to \eqref{eq:BFDequation} as well as several important properties (moments estimates and entropy dissipation). The results of this Appendix are extracted from~\cite{liulu,lu2001spatially,luwennberg}. In a final Appendix~\ref{appendix:FDS}, we provide explicit and uniform-in-$\e$ upper-bounds on quantities related to the Fermi-Dirac statistics relevant to our study.

\subsection*{Acknowledgments} B. L. gratefully acknowledges the financial support from the Italian Ministry of Education, University and Research (MIUR), “Dipartimenti di Eccellenza” grant 2022-2027, as well as the support from the de Castro Statistics Initiative, Collegio Carlo Alberto (Torino). T. B. kindly acknowledges the financial support of the COST Action CA18232 (WG2). We both thank Laurent Desvillettes and Pierre Gervais for the fruitful discussions we had about the problems studied in, and related to, this paper.
\section{\texorpdfstring{Uniform-in-$\e$ $L^{\infty}$ bound: proof of Theorem~\ref{theorem:linftybound}}{Obtention of Linfty bounds}}\label{sec:LpLinf}

This section is devoted to the derivation of $L^\infty$-bounds for solutions to \bfd  which are \emph{uniform in time and independent of $\varepsilon$}, proving Theorem~\ref{theorem:linftybound}. 
Our study is largely inspired from~\cite{alonso2022boltzmann} in which very similar bounds are obtained in the context of the classical Boltzmann equation. The strategy we employ in this paper is to use the fact that a solution to the Boltzmann-Fermi-Dirac equation~\eqref{eq:BFDequation} for some parameter $\e$ is a ``sub-solution'' (in the sense of~\eqref{eq:fsubsolQ0}) to an equation very much resembling the Boltzmann equation (with a modified gain operator), and in particular, free of any dependence in $\e$. This observation allows to completely transpose the estimates and techniques present in~\cite{alonso2022boltzmann} to our problem and obtain without much trouble an $L^{\infty}$ bound on the solution to~\eqref{eq:BFDequation} which is \emph{independent of} $\e$.

\subsection{A link between the Fermi-Dirac and the classical cases}\label{sec:sub}
We present in this section the crucial simplification that allows us to pass from a study on the Boltzmann-Fermi-Dirac operator to a study on the classical Boltzmann operator. This is made possible thanks to the Grad cutoff assumption~\eqref{cutoffassumption}. Recall that, under such an assumption, the classical Boltzmann operator $\Q_{\gamma,b}$ can be split into two parts
\begin{subequations}\label{eq:QQ+Q-}
\begin{equation}\label{eq:Qsplit}
\Q_{\gamma,b}=\Q_{\gamma,b}^{+}-\Q_{\gamma,b}^{-}\end{equation}
where
\begin{equation}\label{eq:Q+-}\begin{cases}
\Q_{\gamma,b}^{+}(f,g)(v)&=\ds \int_{\R^{3}\times\Sb^{2}}f'g'_{*}B(v-\vet,\sigma)\d\vet\d\sigma, \\
\\
\Q_{\gamma,b}^{-}(f,g)(v)&=f(v)\ds\int_{\R^{3}\times\Sb^{2}}g_{*}B(v-v_{*},\sigma)\d \vet.
\end{cases}\end{equation}
\end{subequations}
A similar splitting can be made for $\Q_{\e}$. Indeed, for any $\e > 0$ and $0 \leq f \in L^1_2(\R^3)$ such that $1 - \e f \geq 0,$ defining
\begin{subequations}\label{eq:QQ+Q-e}
\begin{equation}\label{eq:Q+-e}\begin{cases}
\Q^{\e,+}_{\gamma,b}(f,f)&=\ds\iint_{\R^3 \times \Sb^{2}} f' f'_* (1 - \varepsilon f) (1 - \varepsilon f_*) \, B \, \d  v_* \, \d  \sigma\\
\\
\Q^{\e,-}_{\gamma,b}(f,f)&=f \ds \iint_{\R^3 \times \Sb^{2}} f_* (1 - \varepsilon f') (1 - \varepsilon f'_*) \, B \, \d  v_* \, \d  \sigma \,.
\end{cases}\end{equation}
the cutoff assumption~\eqref{cutoffassumption} imply that 
$\Q^{\e,+}_{\gamma,b}(f,f) < \infty$, $\Q^{\e,-}_{\gamma,b}(f,f) < \infty$ and
\begin{equation}\label{eq:Qsplite}
\Q^{\e}_{\g,b}(f,f)=\Q^{\e,+}_{\gamma,b}(f,f)-\Q^{\e,-}_{\gamma,b}(f,f)\,.\end{equation}
\end{subequations}
Moreover, for any measurable and nonngegative $g,h : \R^3 \to \R_+$ and $f \in L^{\infty}(\R^3)$, we define
\begin{equation} \label{eqdef:Q0+Q0bar}
\Gamma_{\gamma,b}(g,h)(v) := \int_{\R^3 \times \Sb^2} g_* (h'+h'_*) \, B(v-v_*,\sigma) \, \d  v_* \, \d  \sigma\,,
\end{equation}
and\begin{equation} \label{eqdef:Qtilde}
\widetilde{\Q}^{+}_{\gamma,b}[f](g,h) := \Q^+_{\gamma,b}(g,h) \, + \, \frac{f}{\|f\|_{\infty}} \, \Gamma_{\gamma,b}(g,h)\,.
\end{equation}
Under the hypothesis that $g,h$ are such that $\Q_{\gamma,b}^{-}(g,h) < \infty$, we can then define
\begin{equation}\label{eq:widetileQ-}
\widetilde{\Q}_{\gamma,b}[f](g,h):=\widetilde{\Q}_{\gamma,b}^{+}[f](g,h)-\Q_{\gamma,b}^{-}(g,h).
\end{equation}
One then has the fundamental proposition where we recall that, in all the paper, assumption~\eqref{cutoffassumption} is in force:
\begin{prop} \label{prop:almostBoltz}
Let $\e > 0$ and $0 \leq f \in L^1_2(\R^3)$ such that $1 - \e f \geq 0$. Then
\begin{equation} \label{eq:comparison}
  \Q_{\gamma,b}^{\e}(f,f) \leq \widetilde{\Q}_{\gamma,b}[f](f,f).
\end{equation}
\end{prop}
\begin{nb} \label{nb:subsol} Proposition~\ref{prop:almostBoltz} implies that any suitable solution to~\eqref{eq:BFDequation} is such that
\begin{equation} \label{eq:fsubsolQ0}
\partial_{t} f \leq \widetilde{\Q}_{\gamma,b}[f](f,f).
\end{equation}
Equation~\eqref{eq:fsubsolQ0} is close to the classical Boltzmann equation, with the difference that the gain part of the collision operator is now $\widetilde{\Q}^{+}_{\gamma,b}$ in place of $\Q_{\gamma,b}^+$, and that it is an inequation. Interestingly, the arguments and techniques present in~\cite{alonso2022boltzmann}, where the classical Boltzmann equation is studied, work even with the differences we mentioned, which will allow us to conclude. The most notable point is of course that~\eqref{eq:fsubsolQ0} is free from any  dependence on $\e$.
\end{nb}
\begin{proof}
Let $\e > 0$ and $0 \leq f \in L^1_2(\R^3)$ such that $1 - \e f \geq 0$. We use the splitting~\eqref{eq:Qsplite} and also set
$$\Upsilon_{\e}:=\e^{2}\iint_{\R^3 \times \Sb^{2}} f f_* f'\,f'_{*}\,B\,\d  v_{*}\,\d \sigma$$
where we notice that $\Upsilon_{\e}$ appears in both $\Q^{\e,+}_{\gamma,b}$ and $\Q^{\e,-}_{\gamma,b}$. Namely, expanding the term $(1-\e f)(1-\e f_{*})$ in $\Q^{\e,+}_{\gamma,b}(f,f)$ one has, using that $f \geq 0$ so that $1-\e f-\e f_* \leq1$,
$$\Q^{\e,+}_{\gamma,b}(f,f) = \iint_{\R^3 \times \Sb^{2}} f' f'_* \, (1 - \e f - \e f_*) \, B \, \d  v_* \, \d  \sigma + \Upsilon_{\e} \leq \Q^{+}_{\gamma,b}(f,f)+\Upsilon_{\e}$$
whereas
\begin{equation*}\begin{split}
\Q^{\e,-}_{\gamma,b}(f,f)&=\iint_{\R^3 \times \Sb^{2}} f f_* \, B \, \d  v_* \, \d  \sigma - \varepsilon \iint_{\R^3 \times \Sb^{2}} f f_*  (f' + f'_*)\, B\, \d  v_* \, \d  \sigma +\Upsilon_{\e}\\
&= \Q^{-}_{\gamma,b}(f,f)- \e f \, \Gamma_{\gamma,b}(f,f)+\Upsilon_{\e}.
\end{split}\end{equation*}
Taking the difference and noticing that $\e f \leq \frac{f}{\|f\|_{\infty}}$ yields the result.
\end{proof}
The crucial property on which our estimates will rely is the fact that the operator $\Gamma_{\gamma,b}$ defined in~\eqref{eqdef:Q0+Q0bar} is adjoint to the (symmetrized version) of $\Q^{+}_{\gamma,b}$ in the following sense:
\begin{lem} \label{lemma:adjointproperty}
For any measurable nonnegative functions $f,g,h$ on $\R^3$, we have
$$
\int_{\R^3} f \, \Gamma_{\gamma,b}(g,h) \, \d  v = \int_{\R^3} h  \left[\Q^+_{\gamma,b}(f,g) + \Q^+_{\gamma,b}(g,f) \right]  \d  v.
$$
\end{lem}
\begin{proof}
Using first the micro-reversibility property of $B$ and then its symmetry property, we get 
\begin{equation*}\begin{split}
\int_{\R^{3}}f \, \Gamma_{\gamma,b}(g,h) \, \d v &= \iiint_{\R^3 \times\R^{3} \times \Sb^2} f g_* (h'+h'_*) \, B \, \d  v \, \d  v_* \, \d  \sigma\\
&= \iiint_{\R^3 \times\R^{3} \times \Sb^2} f' g'_* (h+h_*) \, B \, \d  v \, \d  v_* \, \d  \sigma \\
    &=\iiint_{\R^3 \times\R^{3} \times \Sb^2} f' g'_* h \, B \, \d  v \, \d  v_* \, \d  \sigma + \iiint_{\R^3 \times\R^{3} \times \Sb^2} f'_* g' h \, B \, \d  v \, \d  v_* \, \d  \sigma \\
    &= \int_{\R^3} h  \left[\Q^+_{\gamma,b}(f,g) + \Q^+_{\gamma,b}(g,f) \right]  \d  v\,,\end{split}
    \end{equation*}
    which is the desired result.
\end{proof}

\noindent Lemma~\ref{lemma:adjointproperty} allows to transfer well-known estimates on $\Q_{\gamma,b}^+$ directly to $\Gamma_{\gamma,b}$ and hence to $\widetilde{\Q}_{\gamma,b}^{+}$. The reader may now see how the combination of Proposition~\ref{prop:almostBoltz} and Lemma~\ref{lemma:adjointproperty} allows to easily transfer techniques on the classical Boltzmann operator to the Boltzmann-Fermi-Dirac operator.

\subsection{\texorpdfstring{General estimates on $\Q_{\gamma,b}^+$}{General estimates on Q+}} \label{subsect:genestimatesQ+} 
In this subsection, we recall the already known estimates on $\Q_{\gamma,b}^+$ that we later use in our study. First comes the following proposition, extracted from~\cite[Theorem 2]{alonso2010estimates} (see also  \cite[Theorem 28]{alonso2022boltzmann}) which deals with collision kernel only depending on $b(\cos\theta)$, i.e. associated to $\gamma=0.$
\begin{prop}[\textit{\textbf{Young's type inequality}}] \label{prop:youngineq} Consider an angular kernel $b$ satisfying~\eqref{cutoffassumption}. Let $r \in [1,+\infty]$ and $p,q \in [1,r]$ such that
$$
\frac{1}{p} + \frac{1}{q} = 1 + \frac{1}{r}.
$$
Then
\begin{equation}
\left\|\Q^{+}_{0,b}(g,h)\right\|_{r} \leq  C_{b}(p,q) \; \|g\|_{p}\,\|h\|_{q}, \label{eq:youngQ+}
\end{equation}
where, with $p',q',r'$ respectively the conjugates to $p,q$ and $r$ and $e_1 = (1 \; \; 0 \; \; 0)^T$,
\begin{equation} \label{eqdef:cst_b_pq}
C_b(p,q) = \left( \int_{\Sb^2} \left( \frac{1+e_1 \cdot \sigma}{2}\right)^{-\frac{3}{2 r'}}  \, b(e_1 \cdot \sigma) \, \d \sigma \right)^{\frac{r'}{p'}}  \left( \int_{\Sb^2} \left( \frac{1-e_1 \cdot \sigma}{2}\right)^{-\frac{3}{2 r'}}  \, b(e_1 \cdot \sigma) \, \d \sigma \right)^{\frac{r'}{q'}}.
\end{equation}
In the case $p=q=r=1$, the constant $C_b(1,1)$ is understood as
$$C_b(1,1) = \int_{\Sb^2} b(e_1 \cdot \sigma) \, \d \sigma = \|b\|_{L^1(\Sb^2)},$$
 and in the cases that $p=1$ or $q=1$, one interprets $(\cdot)^{\frac{r'}{p'}} = 1$ and $(\cdot)^{\frac{r'}{q'}}=1$ respectively. As a consequence, for any $p \in [1,+\infty]$ and $f \in L^1(\R^3) \cap L^p(\R^3)$, we have
\begin{equation} \label{eq:YoungbetterQ+}
    \left\|\Q^{+}_{0,b}(f,f)\right\|_{p} \leq  2^{\frac{3}{2 p'}} \, \|b\|_{1} \, \|f\|_{1}\,\|f\|_{p}.
\end{equation}
\end{prop}
\begin{nb} Inequality~\eqref{eq:YoungbetterQ+} has been derived in~\cite[Equation (6.11)]{alonso2022boltzmann} (we point out a misprint in \cite[Eq. (6.11)]{alonso2022boltzmann} where the parameter $s$ should be replaced with $s'$).\end{nb}

One deduces easily the following estimate from Proposition~\ref{prop:youngineq}:
\begin{cor} \label{cor:youngotherQ+L2}
    Consider a bounded angular kernel $b \in L^{\infty}((-1,1))$. Then for $(f,g) \in L^1(\R^3) \times L^2(\R^3)$, we have
    \begin{equation} \label{eq:YoungotherQ+L2}
        \left\|\Q^{+}_{0,b}(f,g)\right\|_{2} \leq 8 \, \|b\|_{{\infty}} \, \|f\|_{1}\,\|g\|_{2}.
    \end{equation}
\end{cor}
\begin{proof}
We use Young's inequality~\eqref{eq:youngQ+} with $p=1$, $q=2$ to obtain, for any $f \in L^1(\R^3),g\in L^2(\R^3)$
$$
\left\| \Q^{+}_{0,b}( f, g) \right\|_{2} \leq C_b(1,2) \, \|f\|_{1} \|g\|_{2}.
$$
Since
$$
C_2(1,2) \leq \|b\|_{\infty} \int_{-1}^1 \left( \frac{1-s}{2} \right)^{-\frac34} \, \d s = 8 \,  \|b\|_{\infty},
$$
we obtain~\eqref{eq:YoungotherQ+L2}.
\end{proof}

\medskip

\noindent The last crucial estimate on $\Q_{\gamma,b}^+$ is the following ``gain of integrability'' property, mostly extracted from~\cite[Equation (6.19)]{alonso2022boltzmann} and adapted to hard potentials.
\begin{prop} \label{prop:gainintegrQ+L2} {(\textit{\textbf{Gain of integrability}}).} Consider $\gamma \in (0,1]$, a bounded angular kernel $b \in L^{\infty}((-1,1))$, $f \in L^1(\R^3)$ and $g \in L^1(\R^3) \cap L^2(\R^3)$. Then
\begin{equation} \label{eq:gainintegrQ+L2}
 \max\left(\left\|\Q^+_{\gamma,b}(g,f) \right\|_{2}\,,\,   \left\|\Q^+_{\gamma,b}(f,g) \right\|_{2}\right) \leq 16 \, \|b\|_{\infty} \, \|f\|_{1}  \,  \|g\|_{1}^{\frac{2 \gamma}{3}} \, \|g\|_{2}^{1 - \frac{2 \gamma}{3}}\,.
\end{equation}
\end{prop}
\color{black}
\begin{proof}
Let $\delta > 0$, notice that for any $u \in \R^3$,
$$
|u|^{\gamma} = |u|^{\gamma}  \, \mathds{1}_{|u| \geq \delta} + |u|^{\gamma} \, \mathds{1}_{|u| \leq \delta} \leq \frac{|u|}{\delta^{1-\gamma}}  \, \mathds{1}_{|u| \geq \delta} + \delta^{\gamma} \, \mathds{1}_{|u| \leq \delta} \leq \frac{|u|}{\delta^{1-\gamma}} + \delta^{\gamma},
$$
so that
$$
\Q^+_{\gamma,b}(f,g) \leq \delta^{\gamma-1} \Q^+_{1,b}(f,g) + \delta^{\gamma}  \Q^+_{0,b}(f,g).
$$
Applying~\cite[Equation (6.19)]{alonso2022boltzmann} (in dimension 3), we have
\begin{equation} \label{eq:gainintegrproof}
\|\Q^+_{1,b}(f,g)\|_{2} \leq C_3 \, \|b\|_{\infty} \, \|f\|_{1} \, \|g\|_{{\frac{6}{5}}},
\end{equation}
where $C_3$ comes from a Hardy-Littlewood-Sobolev inequality (see the last lines of the proof of~\cite[Proposition 30]{alonso2022boltzmann}). Using~\cite[Theorem 3.1, Equation (3.2)]{lieb1983sharp} with here $C_3 = N_{\frac65,1,3}$, we have in fact
$$
C_3 = \pi^{\frac12} \, \frac{\Gamma(\frac32 - \frac12)}{\Gamma(2 - \frac12)} \, \left(\frac{\Gamma(\frac32)}{\Gamma(3)}  \right)^{-1+\frac13}  = \pi^{\frac12} \, \frac{\Gamma(1)}{\Gamma(\frac52)} \, \left(\frac{\Gamma(\frac32)}{\Gamma(3)}  \right)^{-\frac23} = \frac{4^{5/3}}{3 \pi^{\frac13}},
$$
where, only in the above equation, $\Gamma$ stands for the Gamma function. In particular, it holds that $C_3  \leq 3$. Moreover, by interpolation,
$$
\|g\|_{{\frac{6}{5}}} \leq \|g\|_{1}^{\frac23} \, \|g\|_{2}^{\frac13},
$$
therefore~\eqref{eq:gainintegrproof} implies
\begin{equation} \label{eq:gainintegrproof2}
\|\Q^+_{1,b}(f,g)\|_{2} \leq 3 \, \|b\|_{\infty} \, \|f\|_{1} \, \|g\|^{\frac23}_{1} \, \|g\|_{2}^{\frac{1}{3}}.
\end{equation}
On the other hand, it comes from Equation~\eqref{eq:YoungotherQ+L2} in Corollary~\ref{cor:youngotherQ+L2} that
$$
\left\|\Q^{+}_{0,b}(f,g)\right\|_{2} \leq  8 \, \|b\|_{\infty} \, \|f\|_{1}\,\|g\|_{2}.
$$
All in all, we obtain, for any $\delta > 0$,
\begin{align*}
\left\|\Q^{+}_{\gamma,b}(f,g)\right\|_{2} &\leq \delta^{\gamma-1} \left\|\Q^+_{1,b}(f,g)\right\|_{2} + \delta^{\gamma} \left\| \Q^+_{0,b}(f,g)\right\|_{2} \\
&\leq \|b\|_{\infty} \|f\|_{1} \, \|g\|_{2}^{\frac{1}{3}} \left(\delta^{\gamma-1} \, 3 \,  \|g\|_{1}^{\frac{2}{3}}   + \delta^{\gamma} \,  8 \|g\|^{\frac23}_{2}  \right).
\end{align*}
Choosing $\displaystyle \delta = \frac{3 \|g\|_{1}^{\frac23}}{8 \|g\|^{\frac23}_{2}}$ and noting that $3^{\gamma} \times 8^{1-\gamma} \leq 8$ for any $\gamma \in (0,1]$ yields~\eqref{eq:gainintegrQ+L2} with $\|\Q^+_{\g,b}(f,g)\|_2$ in place of the left-hand-side. Let us show now the bound for $\|\Q^+_{\g,b}(g,f)\|_2$. The change of variable $\sigma \to -\sigma$ directly proves that
$$\Q^+_{\g,b}(g,f)=\Q^+_{\g,\tilde{b}}(f,g)$$
with $\tilde{b}(\cos\theta)=b(-\cos\theta)$. Since $\|\tilde{b}\|_\infty=\|b\|_\infty$, we deduce the estimate for $\|\Q^+_{\g,b}(g,f)\|_2$ from the first part of the proof.\end{proof}

\subsection{\texorpdfstring{Key estimates on $\Q^+_{\gamma,b}$, $\Gamma_{\gamma,b}$ and $\Q^-_{\gamma,b}$}{Estimates on Q+, Gamma and Q-}} \label{subsect:estimates} 

\subsubsection{\texorpdfstring{Estimates on $\Q^+_{\gamma,b}$}{Estimates on Q+}} 
Relying on the general estimates on $\Q^+_{\gamma,b}$ stated above, we provide here two useful bounds that we later use in order to obtain $L^2_{\gamma}$ bounds on the solutions to~\eqref{eq:BFDequation}.

\begin{prop} \label{prop:estimatesQ+1}
Consider a bounded angular kernel $b \in L^{\infty}((-1,1))$ and $\gamma \in (0,1]$. Then for any $f \in L^1_{3 \gamma}(\R^3) \cap L^{\infty}(\R^3)$, we have
\begin{equation} \label{eq:L2boundQ+bLinfty}
    \int_{\R^3} \Q^+_{\gamma,b}(f,f) f(v) \, \langle v \rangle^{2 \gamma} \, \d v \leq 16 \, \|b\|_{\infty} \, \|\langle \cdot \rangle^{\frac{\gamma}{2}}f\|_{1}^{1+\frac{2 \gamma}{3}} \, \|\langle \cdot \rangle^{\frac{3\gamma}{2}}f\|_{2}^{2 - \frac{2 \gamma}{3}}.
\end{equation}

\end{prop}

\medskip

\begin{proof}
First, by Cauchy-Schwarz inequality, we have
$$
\int_{\R^3} \Q_{\gamma,b}^+(f,f) \, f(v) \, \langle v \rangle^{2\gamma} \, \d v \leq \|\langle \cdot \rangle^{\frac{3\gamma}{2}}f\|_{2} \, \left\| \langle \cdot \rangle^{\frac{\gamma}{2}}\Q_{\gamma,b}^+(f,f) \right\|_{2}. 
$$
Now notice that for any $(v,v_*,\sigma) \in \R^3 \times \R^3 \times \Sb^2$ we have $\langle v \rangle^{\frac{\gamma}{2}} \leq \langle v' \rangle^{\frac{\gamma}{2}} \langle v'_* \rangle^{\frac{\gamma}{2}}$, where we recall the notation~\eqref{notationvprime} for $v'$ and $v'_*$. From this we deduce that for any $v \in \R^3$, we have
$$
\langle v \rangle^{\frac{\gamma}{2}} \, \Q_{\gamma,b}^+(f,f)(v) \leq  \Q_{\gamma,b}^+(\langle \cdot \rangle^{\frac{\gamma}{2}}f, \, \langle \cdot \rangle^{\frac{\gamma}{2}}f)(v).
$$
Lastly, we use the gain of integrability property~\eqref{eq:gainintegrQ+L2} to obtain that
$$
\left\| \Q_{\gamma,b}^+(\langle \cdot \rangle^{\frac{\gamma}{2}}f,\langle \cdot \rangle^{\frac{\gamma}{2}}f) \right\|_{2} \leq  16 \, \|b\|_{\infty} \, \|\langle \cdot \rangle^{\frac{\gamma}{2}}f\|_{1}^{1+\frac{2 \gamma}{3}}   \, \|\langle \cdot \rangle^{\frac{\gamma}{2}}f\|_{2}^{1 - \frac{2 \gamma}{3}}.
$$
Since $\|\langle \cdot \rangle^{\frac{\gamma}{2}}f\|_{2} \leq \|\langle \cdot \rangle^{\frac{3\gamma}{2}}f\|_{2}$ and $1 - \frac{2 \gamma}{3} \geq 0$, we indeed obtain~\eqref{eq:L2boundQ+bLinfty}.
\end{proof}

We consider in the next result an estimate for $\Q^+_{\g,b}$, similar to Young's convolution inequality, under the mere cutoff assumption~\eqref{cutoffassumption}: 
\begin{prop} \label{prop:estimatesQ+2}
Consider an angular kernel $b$ satisfying the cutoff assumption~\eqref{cutoffassumption} and $\gamma \in (0,1]$. Then for any $f \in L^1_{3 \gamma}(\R^3) \cap L^{\infty}(\R^3)$, we have
\begin{equation} \label{eq:L2boundQ+bL1}
    \int_{\R^3} \Q^+_{\gamma,b}(f,f) f(v) \, \langle v \rangle^{2 \gamma} \, \d v \leq  2^{\frac34} \, \|b\|_{1} \, \|\langle \cdot \rangle^{\frac{3\gamma}{2}}f\|_{1}  \, \|\langle \cdot \rangle^{\frac{3\gamma}{2}}f\|^2_{2}.
\end{equation}
\end{prop}
\begin{proof}
We follow the same steps as in the proof of the previous Proposition~\ref{prop:estimatesQ+1} and recall that
$$
\int_{\R^3} \Q^+_{\gamma,b}(f,f) f(v) \, \langle v \rangle^{2\gamma} \, \d v \leq \|\langle \cdot \rangle^{\frac{3\gamma}{2}}f\|_{2} \, \left\| \langle \cdot \rangle^{\frac{\gamma}{2}}\Q_{\gamma,b}^+(f,f) \right\|_{2}.
$$
Second, notice that for any $(v,v_*,\sigma) \in \R^3 \times \R^3  \times \Sb^2$ it holds that
$$
|v-v_*|^{\gamma} = |v'-v'_*|^{\gamma} \leq \langle v' \rangle^{\gamma} \langle v'_* \rangle^{\gamma},
$$
and
$\langle v \rangle^{\frac{\gamma}{2}} \leq \langle v' \rangle^{\frac{\gamma}{2}} \langle v'_* \rangle^{\frac{\gamma}{2}}$, where we recall the notation~\eqref{notationvprime} for $v'$ and $v'_*$. Then, for any $v \in \R^3$,
$$
\langle v \rangle^{\frac{\gamma}{2}}\Q_{\gamma,b}^+(f,f)(v) \leq \langle v \rangle^{\frac{\gamma}{2}}\Q_{0,b}^+(\langle \cdot \rangle^{\gamma} f, \, \langle \cdot \rangle^{\gamma} f)(v) \leq \Q_{0,b}^+(\langle \cdot \rangle^{\frac{3\gamma}{2}} f, \, \langle \cdot \rangle^{\frac{3\gamma}{2}} f)(v).
$$
Since~\eqref{eq:YoungbetterQ+} in Proposition~\ref{prop:youngineq} gives
\begin{align*}
\left\|\Q^+_{0,b}(\langle \cdot \rangle^{\frac{3 \gamma}{2}}f, \langle \cdot \rangle^{\frac{3 \gamma}{2}}f) \right\|_{2} \leq 2^{\frac34} \, \|b\|_{1} \, \|\langle \cdot \rangle^{\frac{3\gamma}{2}}f\|_{1}  \, \|\langle \cdot \rangle^{\frac{3\gamma}{2}}f\|_{2},
\end{align*}
we can conclude to~\eqref{eq:L2boundQ+bL1}.
\end{proof}

\noindent Finally, we provide the following proposition, which is a straightforward consequence of~\cite[Corollary 35]{alonso2022boltzmann}.
\begin{prop} \label{prop:bsbRQ+}
Consider angular kernels $b_S$ and $b_R$ satisfying the cutoff assumption~\eqref{cutoffassumption} and such that $b_S$ vanishes in the vicinity of $\{-1,1\}$. Then for any $\gamma \in (0,1]$ and $f \in L^1_{2 \gamma}(\R^3) \cap L^\infty(\R^3)$, we have
\begin{equation} \label{eq:bsQ+}
\left\|\Q^{+}_{\gamma,b_S}(f,f)\right\|_{\infty} \leq  C_{b_S}(2,2) \, \|\langle \cdot \rangle^{ \gamma}f\|_{2}^2,
\end{equation}
where $C_{b_S}(2,2)$ is defined by~\eqref{eqdef:cst_b_pq}, and for any $v \in \R^3$,
\begin{equation} \label{eq:bRQ+}
    \Q^{+}_{\gamma,b_R}(f,f) (v) \leq  4 \, \|b_R\|_{1} \, \|f\|_{1}\,\|f\|_{\infty} \, \langle v \rangle^{\gamma}.
\end{equation}
\end{prop}

\begin{proof}
The only difference with~\cite[Corollary 35]{alonso2022boltzmann} is that we explicit the constant $4\, \|b_R\|_{1}$ in~\eqref{eq:bRQ+} and that we have added up the two last equations of ~\cite[Corollary 35, (7.29)]{alonso2022boltzmann}. The $4$ is not explicitly given in~\cite[Corollary 35]{alonso2022boltzmann} but can be obtained by noting that the constant is there given by $2^{\frac{\gamma}{2}} (C_{b_R^+}(\infty,1) + C_{b_R^-}(1,\infty) )$ (where $b_R^+$ and $b^-_R$ are defined there), that one can obtain from~\eqref{eqdef:cst_b_pq}, with
$$
C_{b_R^+}(\infty,1) \leq 2^{\frac32} \|b^+_R\|_{1}, \qquad C_{b_R^-}(1,\infty) \leq 2^{\frac32} \|b^-_R\|_{1},
$$
hence allowing to bound the constant by $2^{\frac{\gamma + 3}{2}}\, \|b_R\|_{1}$. We then conclude as $\gamma \leq 1$.
\end{proof}

\subsubsection{\texorpdfstring{Estimates on $\Gamma_{\gamma,b}$}{Estimates on Gamma}}
In this subsection, we make use of the estimates on $\Q^+_{\gamma,b}$ obtained in the previous subsection as well as Lemma~\ref{lemma:adjointproperty} to simply obtain crucial estimates on $\Gamma_{\gamma,b}$.

\medskip

\noindent Propositions~\ref{prop:gamma1}--\ref{prop:gamma2} and~\ref{prop:bsGamma}--\ref{prop:bRGamma} are the analogues to Propositions~\ref{prop:estimatesQ+1}--\ref{prop:estimatesQ+2} and~\ref{prop:bsbRQ+} for $\Gamma_{\gamma,b}$.

\begin{prop} \label{prop:gamma1}
Consider a bounded angular kernel $b \in L^{\infty}((-1,1))$. Then for any $\gamma \in (0,1]$ and $0 \leq f \in L^1_{3 \gamma}(\R^3) \cap L^{\infty}(\R^3)$, we have
\begin{equation} \label{eq:L2boundGammabLinfty}
\int_{\R^3} \frac{f(v)}{\|f\|_{\infty}} \, \Gamma_{\gamma,b}(f,f) f(v) \, \langle v \rangle^{2\gamma} \, \d v \leq 32 \, \|b\|_{\infty} \,\left\|\langle \cdot \rangle^{\gamma}f \right\|_{1}^{1 + \frac{2 \gamma}{3}}  \, \|\langle \cdot \rangle^{\frac{3\gamma}{2}}f\|_{2}^{2 - \frac{2 \gamma}{3}}.
\end{equation}
\end{prop}
\begin{proof}
First, notice that applying the Cauchy-Schwarz inequality yields
\begin{equation} \label{eq1propGamma}
\int_{\R^3} \frac{f(v)}{\|f\|_{\infty}} \, \Gamma_{\gamma,b}(f,f) f(v) \, \langle v \rangle^{2\gamma} \, \d v \leq \|f\|_{\infty}^{-1} \, \|\langle \cdot \rangle^{\frac{3\gamma}{2}}f\|_{2} \, \left\| \langle \cdot \rangle^{\frac{\gamma}{2}} \, f \, \Gamma_{\gamma,b}(f,f) \right\|_{2}.
\end{equation}
Clearly, by Riesz representation theorem, $$
\left\| \langle \cdot \rangle^{\frac{\gamma}{2}} \, f \, \Gamma_{\gamma,b}(f,f) \right\|_{2} =  \sup_{\|g\|_2 = 1} \int_{\R^3} g(v)\, \Gamma_{\gamma,b}(f,f) f(v) \, \langle v \rangle^{\frac{\gamma}{2}} \, \d v.
$$
Applying Lemma~\ref{lemma:adjointproperty}, this gives 
\begin{equation*}
\begin{split}
    \left\| \langle \cdot \rangle^{\frac{\gamma}{2}} \, f \, \Gamma_{\gamma,b}(f,f) \right\|_{2} &=
\sup_{\|g\|_2 = 1} \int_{\R^3}  f(v) \left\{ \Q_{\gamma,b}^+(\langle \cdot \rangle^{\frac{\gamma}{2}} f g, f) + \Q_{\gamma,b}^+(f,\langle \cdot \rangle^{\frac{\gamma}{2}} f g) \right\} \d v\\
&\leq \| f \|_{2} \,  \sup_{\|g\|_2 = 1}  \left\{ \left\|\Q_{\gamma,b}^+(\langle \cdot \rangle^{\frac{\gamma}{2}} f g, f) \right\|_{2} + \left\| \Q_{\gamma,b}^+(f,\langle \cdot \rangle^{\frac{\gamma}{2}} f g) \right\|_{2}\right\}\,.
\end{split}    
\end{equation*}
where we used Cauchy-Schwarz inequality for the last estimate. Using the gain of integrability property~\eqref{eq:gainintegrQ+L2} in Proposition~\ref{prop:gainintegrQ+L2}, we deduce that
$$ \left\| \langle \cdot \rangle^{\frac{\gamma}{2}} \, f \, \Gamma_{\gamma,b}(f,f) \right\|_{2} \leq
\| f \|_{2} \,  \sup_{\|g\|_2 = 1}  \left\{ 32 \, \|b\|_{\infty}\, \left\|\langle \cdot\rangle^{\frac{\gamma}{2}} f g  \right\|_{1} \, \left\|f \right\|_{1}^{\frac{2 \gamma}{3}} \, \, \left\|f \right\|_{2}^{1 - \frac{2 \gamma}{3}} \right\}.
$$
Using again Cauchy-Schwarz inequality, we obtain
\begin{multline*}
\left\| \langle \cdot \rangle^{\frac{\gamma}{2}} \, f \, \Gamma_{\gamma,b}(f,f) \right\|_{2}  \leq  \| f \|_{2} \,  \sup_{\|g\|_2 = 1}  \left\{ 32 \, \|b\|_{\infty}\, \left\|\langle \cdot \rangle^{\frac{\gamma}{2}}f  \right\|_{2} \, \|g\|_{2} \,  \left\|f \right\|_{1}^{\frac{2 \gamma}{3}} \, \, \left\|f \right\|_{2}^{1 - \frac{2 \gamma}{3}} \right\} \\
=\; 32 \, \|b\|_{\infty} \, \left\|f \right\|_{1}^{\frac{2 \gamma}{3}}  \,  \left\|\langle \cdot \rangle^{\frac{\gamma}{2}}f  \right\|_{2} \, \left\|f \right\|_{2}^{2 - \frac{2 \gamma}{3}}.  
\end{multline*} 
Recalling~\eqref{eq1propGamma}, we then conclude that
\begin{multline*}
\int_{\R^3} \frac{f(v)}{\|f\|_{\infty}} \, \Gamma_{\gamma,b}(f,f) f(v) \, \langle v \rangle^{2\gamma} \, \d v \leq
32  \|b\|_{\infty}\|f\|_{\infty}^{-1} \left\|\langle \cdot \rangle^{\frac{3\gamma}{2}}f  \right\|_{2}   \left\|f \right\|_{1}^{\frac{2 \gamma}{3}}  \,  \left\|\langle \cdot\rangle^{\frac{\gamma}{2}}f  \right\|_{2} \, \left\|f \right\|_{2}^{2 - \frac{2 \gamma}{3}}.
\end{multline*}
The result comes after noticing that
$$
\| f \|_{2} \leq \|f\|_{1}^{\frac12}\|f\|_{\infty}^{\frac{1}{2}}, \quad \|\langle\cdot\rangle^{\frac{\gamma}{2}}f \|_{2} \leq \|\langle\cdot\rangle^{\gamma}f\|_{1}^{\frac12}\,\|f\|_{\infty}^{\frac{1}{2}},$$
while $\|f\|_{1} \leq \|\langle \cdot \rangle^{\gamma} f\|_{1}$ and $\left\|f \right\|_{2} \leq \left\|\langle \cdot \rangle^{\frac{3\gamma}{2}}f \right\|_{2}.$\end{proof}

\medskip

\noindent For $b$ satisfying the cutoff assumption~\eqref{cutoffassumption}, we have the following version of Young's convolution inequality
\begin{prop} \label{prop:gamma2}
Consider $\gamma \in (0,1]$, an angular kernel $b$ satisfying the cutoff assumption~\eqref{cutoffassumption} and $f \in L^1_2(\R^3) \cap L^{\infty}(\R^3)$. Then
\begin{equation} \label{eq:L2boundGammabL1}
    \int_{\R^3} \frac{f}{\|f\|_{\infty}} \, \Gamma_{\gamma,b}(f,f) f(v) \, \langle v \rangle^{2\gamma} \, \d v \leq  2 \|b\|_{1} \, \|\langle \cdot \rangle^{\gamma} f\|_{1} \,\, \|\langle \cdot \rangle^{\frac{3\gamma}{2}}f\|^{2}_{2}.
\end{equation}
\end{prop}
\begin{proof}
We first note that, since $|v-v_*|^{\gamma} \leq \langle v \rangle^{\gamma} \langle v_* \rangle^{\gamma}$, where we recall that $\langle v \rangle = (1+|v|^2)^{\frac12}$, we have
\begin{equation}\label{eq:GaGb}\begin{split}
    \int_{\R^3} \frac{f}{\|f\|_{\infty}} \, \Gamma_{\gamma,b}(f,f) f(v) \, \langle v \rangle^{2\gamma} \, \d v &\leq \|f\|_{\infty}^{-1} \,  \int_{\R^3} \Gamma_{0,b}(\langle \cdot \rangle^{\gamma}f,f) f^2(v) \, \langle v \rangle^{3\gamma} \, \d v\\
    &\leq \|f\|_{\infty}^{-1} \, \left\| \Gamma_{0,b}(\langle \cdot \rangle^{\gamma}f,f)  \right\|_{\infty} \, \|\langle \cdot \rangle^{\frac{3\gamma}{2}}f\|^2_{2}.\end{split}
\end{equation}
As before, thanks to Riesz representation theorem together with Lemma ~\ref{lemma:adjointproperty}, we have
\begin{equation*}\begin{split}
\left\| \Gamma_{0,b}(\langle \cdot \rangle^{\gamma}f,f)  \right\|_{\infty} &= \sup_{\|g\|_{1} = 1} \int_{\R^3} g(v) \, \Gamma_{0,b}(\langle \cdot \rangle^{\gamma}f,f) \, \d v\\ 
&=\sup_{\|g\|_{1} = 1} \int_{\R^3} f(v) \left(\Q^+_{0,b}(g,\langle \cdot \rangle^{\gamma}f) + \Q^+_{0,b}(\langle \cdot \rangle^{\gamma}f,g) \right) \d v.
\end{split}\end{equation*}
Clearly, this implies that
$$
\left\| \Gamma_{0,b}(\langle \cdot \rangle^{\gamma}f,f)  \right\|_{\infty}\leq\|f\|_{\infty} \, \sup_{\|g\|_{1} = 1} \left( \|\Q^+_{0,b}(g,\langle \cdot \rangle^{\gamma}f)\|_{1} + \|\Q^+_{0,b}(\langle \cdot \rangle^{\gamma}f,g)\|_{1} \right).
$$
We apply Young's inequality~\eqref{eq:youngQ+} from Proposition~\ref{prop:youngineq} with $(p,q,r) = (1,1,1)$ to deduce that 
$$
\left\| \Gamma_{0,b}(\langle \cdot \rangle^{\gamma}f,f)  \right\|_{\infty} \leq \|f\|_{\infty} \, \sup_{\|g\|_{1} = 1} \left( 2 \|b\|_{1} \|g\|_{1} \, \|\langle \cdot \rangle^{\gamma}f\|_{1} \right) = 2 \|b\|_{1}  \|f\|_{\infty} \|\langle \cdot \rangle^{\gamma}f\|_{1} .
$$
Combining this estimate with~\eqref{eq:GaGb} yields~\eqref{eq:L2boundGammabL1}.
\end{proof}

\begin{prop} \label{prop:bsGamma} 
Consider an angular kernel $b$ satisfying the cutoff assumption~\eqref{cutoffassumption} and such that $b$ vanishes in the vicinity of $\{-1,1\}$. Then for any $\gamma \in (0,1]$ and $f \in L^2_{\gamma}(\R^3) \cap L^\infty(\R^3)$, we have for any $v \in \R^3$,
\begin{equation} \label{eq:bsGamma}
\frac{f(v)}{\|f\|_{\infty}} \, \Gamma_{\gamma,b}(f,f) (v) \leq  \left( C_{b}(1,2) + C_{b}(2,1) \right)  \; \|\langle \cdot \rangle^{\gamma}f\|_{2}  \,\|f\|_{2} \, \langle v \rangle^{\gamma},
\end{equation}
where $C_{b}(1,2) $ and $ C_{b}(2,1)$ are defined in~\eqref{eqdef:cst_b_pq}.
\end{prop}

\begin{proof}
For $g \in L^2(\R^3)$, writing the $L^\infty$ norm in weak form yields
$$
\|\Gamma_{0,b}(g,f)\|_{\infty} = \sup_{\|h\|_{1}=1} \int_{\R^3} h(v) \, \Gamma_{0,b}(g,f)(v) \, \d v.
$$
Applying Lemma~\ref{lemma:adjointproperty}, the above right-hand-side becomes
$$
\sup_{\|h\|_{1} = 1} \int_{\R^3}  f(v) \left\{ \Q_{0,b}^+(h,g)(v) + \Q_{0,b}^+(g,h)(v) \right\} \d v,
$$
which, by a Cauchy-Schwarz argument, is lower than
$$
\|f\|_{2} \sup_{\|h\|_{1} = 1} \left(\|\Q_{0,b}^+(h,g) \|_{2} + \| \Q_{0,b}^+(g,h) \|_{2} \right).
$$
We apply Young's inequality~\eqref{eq:youngQ+} from Proposition~\ref{prop:youngineq} respectively with $(p,q,r) = (1,2,2)$ and $(p,q,r) = (2,1,2)$ to obtain that the above term is lower than
$$
\|f\|_{2} \, \sup_{\|h\|_{1} = 1}  \left( C_{b}(1,2) + C_{b}(2,1) \right) \|g\|_{2} \, \|h\|_{1},
$$
so that
$$
\|\Gamma_{0,b}(g,f)\|_{\infty} \leq   \left( C_{b}(1,2) + C_{b}(2,1) \right) \|g\|_{2} \, \|f\|_{2}.
$$
Finally, we recall that, for any $(v,v_*,\sigma) \in \R^3 \times \R^3 \times \Sb^2$, we have $|v-v_*|^{\gamma} \leq \langle v \rangle^{\gamma} \langle v_* \rangle^{\gamma}$, so that
$$
\Gamma_{\gamma,b}(g,f)(v) \leq \Gamma_{0,b}(\langle \cdot \rangle^{\gamma} \, g, f)(v) \, \langle v\rangle^{\gamma},
$$
ending the proof, after noticing that $\frac{f(v)}{\|f\|_{\infty}} \leq 1$.
\end{proof}

\begin{prop} \label{prop:bRGamma} 
Consider an angular kernel $b$ satisfying the cutoff assumption~\eqref{cutoffassumption}. Then for any $\gamma \in (0,1]$ and $f \in L^1_{\gamma}(\R^3) \cap L^\infty(\R^3)$, we have for any $v \in \R^3$,
\begin{equation} \label{eq:bRGamma}
\frac{f(v)}{\|f\|_{\infty}} \, \Gamma_{\gamma,b}(f,f) (v) \leq   2 \, \|b\|_{1} \, \|f\|_{L^1_{\gamma}} \, f(v) \, \langle v \rangle^{\gamma}.
\end{equation}
\end{prop}

\begin{proof}
For $g \in L^1(\R^3)$, we showed in the first lines of the proof of the previous proposition that
$$
\|\Gamma_{0,b}(g,f)\|_{\infty} \leq \sup_{\|h\|_{1} = 1} \int_{\R^3}  f(v) \left\{ \Q_{0,b}^+(h,g)(v) + \Q_{0,b}^+(g,h)(v) \right\} \d v\,.
$$
Therefore, $$\|\Gamma_{0,b}(g,f)\|_{\infty} \leq
\|f\|_{\infty} \sup_{\|h\|_{1} = 1} \left(\|\Q_{0,b}^+(h,g) \|_{1} + \| \Q_{0,b}^+(g,h) \|_{1} \right).
$$
We apply Young's inequality~\eqref{eq:youngQ+} from Proposition~\ref{prop:youngineq} with $(p,q,r) = (1,1,1)$ to obtain that 
$$
\|\Gamma_{0,b}(g,f)\|_{\infty} \leq   2 \, \|b\|_{1} \, \|g\|_{1} \, \|f\|_{\infty}.
$$
Finally, we recall that, for any $(v,v_*,\sigma) \in \R^3 \times \R^3 \times \Sb^2$, we have $|v-v_*|^{\gamma} \leq \langle v \rangle^{\gamma} \langle v_* \rangle^{\gamma}$, so that
$$
\Gamma_{\gamma,b}(g,f)(v) \leq \Gamma_{0,b}(\langle \cdot \rangle^{\gamma} \, g, f)(v) \, \langle v\rangle^{\gamma},
$$
ending the proof.
\end{proof}

\subsubsection{\texorpdfstring{Estimate on $\Q^-_{\gamma,b}$}{Estimates on Q-}}

Lastly, we use a standard lower-bound estimate on the loss operator~$\Q^-_{\gamma,b}$.

\begin{prop} \label{prop:lwrbdQ-}
Consider $\gamma \in (0,1]$ and an angular kernel $b$ satisfying the cutoff assumption~\eqref{cutoffassumption}. Then for any $0 \leq f \in L^1_3(\R^3)$, there exists an explicit constant $\mathbf{c}_{\gamma, f} > 0$ depending on $\gamma$ and on $f$ only through its $L^1$ and $L^1_2$ norms and an upper-bound on its $L^1_3$ norm, such that
\begin{equation} \label{eq:lowerbdQ-}
\int_{\R^3} \Q^-_{\gamma,b}(f,f) \, f(v) \, \langle v \rangle^{2 \gamma} \, \d v \geq \mathbf{c}_{\gamma,f} \, \|b\|_{1}  \,  \|\langle \cdot \rangle^{\frac{3\gamma}{2}}f\|^2_{2}.
\end{equation}
\end{prop}

\begin{proof}
First note that the loss operator may be written as, for any $v \in \R^3$,
$$
\Q^-_{\gamma,b}(f,f)(v) = f(v) \, \|b\|_{1} \, (f \ast |\cdot|^{\gamma})(v).
$$
Then~\cite[Lemma 8]{alonso2022boltzmann} provides the existence of an explicit $\mathbf{c}_{\gamma, f} > 0$ depending on $\gamma$ and on $f$ only through $\|f\|_{1}$, $\|f\|_{L^1_2}$ and an upper bound on $\|f\|_{L^1_3}$ such that for any $v \in \R^3$,
\begin{equation}\label{eq:loweQQ}
(f \ast |\cdot|^{\gamma})(v) \geq \mathbf{c}_{\gamma, f}\, \langle v \rangle^{\gamma}.
\end{equation}
Combing those two fact, it then comes that
$$
\int_{\R^3} \Q^-_{\gamma,b}(f,f) \, f(v) \, \langle v \rangle^{2 \gamma} \, \d v  \geq \mathbf{c}_{\gamma, f} \, \|b\|_{1} \int_{\R^3}  f(v)^2 \, \langle v \rangle^{3 \gamma} \, \d v,
$$
concluding the proof.
\end{proof}

\subsection{\texorpdfstring{Derivation of  $L^2_{\gamma}$ bounds}{Obtention of an L2 bound}}

In this subsection, we provide, in Proposition~\ref{prop:L2bound}, a uniform-in-time and uniform-in-$\e$ $L^2_{\gamma}$ bound on the solutions to~\eqref{eq:BFDequation}. We only recall that, according to the results in Appendix~\ref{app:cauchy}, assuming that $f^{\rm in} \in L^1_3(\R^3) \cap L^{\infty}(\R^3)$, for any $0 < \varepsilon \leq \|f^{\rm in}\|_{\infty}^{-1}$, the unique solution $f^{\varepsilon}$ to \bfd  is such that 

\begin{equation} \label{eq:supL13}
    \sup_{t\geq 0} \|f^{\varepsilon}(t,\cdot)\|_{L^1_3} \leq \mathbf{C}_{1,3},
\end{equation}
with $\mathbf{C}_{1,3} > 0$ explicit and depending only on $\gamma$, $b$, $\varrho^{\rm in}$, $u^{\rm in}$, $E^{\rm in}$ and an upper-bound on $\|f^{\rm in}\|_{L^1_3}$ - in particular, not on $\e$.

\medskip

\noindent We now have all the tools we need to obtain the following proposition.

\begin{prop} \textbf{\textit{($L^2_{\gamma}$ bound.)}} \label{prop:L2bound}
Consider a collision kernel $B$ of the form~\eqref{eq:HypBB} with $\gamma \in (0,1]$. Given an initial datum $0 \leq f^{\rm in} \in L^1_3(\R^3) \cap L^\infty(\R^3)$, there exists an explicit constant $\mathbf{C}_{2,\gamma}$ depending only on $\gamma$, $b$, $\varrho^{\rm in}$, $u^{\rm in}$, $E^{\rm in}$ and upper-bounds on $\|f^{\rm in}\|_{L^1_3}$ and $\|f^{\rm in}\|_{\infty}$ such that for any $\e \in (0,\|f^{\rm in}\|^{-1}_{\infty}]$, the unique solution $f^{\varepsilon}$ to~\eqref{eq:BFDequation} satisfies
\begin{equation}
\sup_{t \geq 0} \, \|f^{\varepsilon}(t)\|_{L^2_{\gamma}} \leq  \mathbf{C}_{2,\gamma}.
\end{equation}
\end{prop}

\begin{proof}
For the sake of clarity, since we are working here with fixed $\varepsilon$ and fixed initial datum, we simply denote in the following $f(t) \equiv f^\e(t,\cdot)$ for any $t \geq 0$ as the unique solution to 
$$
\partial_t f(t)(v) = \Q^{\e}_{\gamma,b}(f(t),f(t))(v), \qquad \qquad f(t=0) = f^{\rm in}.
$$
Using Proposition~\ref{eq:comparison} we get, as pointed out in Remark~\ref{nb:subsol} together with  the definition of $\widetilde{\Q}_{\gamma,b}$ (see Eq.~\eqref{eq:widetileQ-}), we deduce that, for any $t \geq 0$,
$$
\partial_t f(t) \leq \widetilde{\Q}_{\gamma,b}(f(t),f(t)) = \Q^+_{\gamma,b}(f(t),f(t)) + \frac{f(t)}{\|f(t)\|_{\infty}}\Gamma_{\gamma,b}(f(t),f(t)) - \Q^-_{\gamma,b}(f(t),f(t)).
$$
Multiplying both sides by $f(t)(v) \, \langle v \rangle^{2\gamma}$ and integrating on $v \in \R^3$, it comes that
\begin{multline*}
   \frac12 \frac{\d}{\d t} \|f(t)\|^2_{L^2_{\gamma}} \leq \int_{\R^3} \bigg(\Q^+_{\gamma,b}(f(t),f(t))(v) +  \frac{f(t)(v)}{\|f(t)\|_{\infty}} \, \Gamma_{\gamma,b}(f(t),f(t))(v) \\
   -  \Q^-_{\gamma,b}(f(t),f(t))(v) \bigg)  f(t)(v) \, \langle v \rangle^{2 \gamma} \, \d v. 
\end{multline*}
 In the rest of the proof, we follow the ideas developed in~\cite[Section 7]{alonso2022boltzmann} (see also~\cite{alonso2017}) and  split the kernel $b$ into $b =  b^{\infty} + b^1$, with $b^{\infty} \in L^{\infty}((-1,1))$. Since $b \mapsto \Q^+_{\gamma,b}$ and $b \mapsto \Gamma_{\gamma,b}$ are linear applications, we can split the above right-hand-side as
\begin{align}
  \frac12 \frac{\d}{\d t} \|f(t)\|^2_{L^2_{\gamma}} \leq \; \; \; \;  &  \int_{\R^3}\Q^+_{\gamma,b^{\infty}}(f(t),f(t))(v) \,  f(t)(v) \, \langle v \rangle^{2 \gamma} \, \d v \label{eqproofL2-1} \\
  + \; &\int_{\R^3}  \Q^+_{\gamma,b^{1}}(f(t),f(t))(v) \, f(t)(v) \, \langle v \rangle^{2 \gamma} \, \d v \label{eqproofL2-2} \\
  + \; &\int_{\R^3}   \frac{f(t)(v)}{\|f(t)\|_{\infty}} \, \Gamma_{\gamma,b^{\infty}}(f(t),f(t))(v) \, f(t)(v) \, \langle v \rangle^{2 \gamma} \, \d v \label{eqproofL2-3} \\
  + \; &\int_{\R^3}   \frac{f(t)(v)}{\|f(t)\|_{\infty}} \, \Gamma_{\gamma,b^{1}}(f(t),f(t))(v) \, f(t)(v) \, \langle v \rangle^{2 \gamma} \, \d v \label{eqproofL2-4} \\
  - \; &\int_{\R^3} \Q^-_{\gamma,b}(f(t),f(t))(v) \,  f(t)(v) \, \langle v \rangle^{2 \gamma} \, \d v. \label{eqproofL2-5}
\end{align}
All of these terms have been studied and bouned from above in Subsection~\ref{subsect:estimates}. Namely:
\begin{itemize}
    \item Equation~\eqref{eq:L2boundQ+bLinfty} in Proposition~\ref{prop:estimatesQ+1} allows to bound from above~\eqref{eqproofL2-1}
    \item Equation~\eqref{eq:L2boundQ+bL1} in Proposition~\ref{prop:estimatesQ+2} allows to bound from above~\eqref{eqproofL2-2}
    \item Equation~\eqref{eq:L2boundGammabLinfty} in Proposition~\ref{prop:gamma1} allows to bound from above~\eqref{eqproofL2-3}
    \item Equation~\eqref{eq:L2boundGammabL1} in Proposition~\ref{prop:gamma2} allows to bound from above~\eqref{eqproofL2-4}
    \item Equation~\eqref{eq:lowerbdQ-} in Proposition~\ref{prop:lwrbdQ-} allows to bound from above~\eqref{eqproofL2-5}.
\end{itemize}
Putting everything together,~\eqref{eqproofL2-1}--~\eqref{eqproofL2-5} implies
\begin{equation*}\begin{split}
    \frac12 \frac{\d}{\d t} \|f(t)\|^2_{L^2_{\gamma}}+\mathbf{c}_{\gamma,f(t)} \, \|b\|_{1}  \,  \|\langle \cdot \rangle^{\frac{3\gamma}{2}}f(t)\|^2_{2} &\leq 16 \, \|b^{\infty}\|_{\infty} \, \|\langle \cdot\rangle^{\frac{\gamma}{2}}f(t)\|_{1}^{1+\frac{2 \gamma}{3}} \, \|\langle \cdot \rangle^{\frac{3\gamma}{2}}f(t)\|_{2}^{2 - \frac{2 \gamma}{3}} \\
&+ 2^{\frac34} \, \|b^1\|_{1} \, \|\langle \cdot \rangle^{\frac{3\gamma}{2}}f(t)\|_{1}  \, \|\langle \cdot \rangle^{\frac{3\gamma}{2}}f(t)\|^2_{2}\\
&+ 32 \, \|b^{\infty}\|_{\infty} \,\left\|\langle \cdot \rangle^{\gamma}f(t) \right\|_{1}^{1 + \frac{2 \gamma}{3}}  \, \|\langle \cdot \rangle^{\frac{3\gamma}{2}}f(t)\|_{1}^{2 - \frac{2 \gamma}{3}} \\
&+2 \|b^1\|_{1} \, \|\langle \cdot \rangle^\gamma f(t)\|_{1} \,\, \|\langle \cdot \rangle^{\frac{3\gamma}{2}}f(t)\|^{2}_{2}\,.
\end{split}\end{equation*}
We observe now  that $2^{\frac34} \leq 2$ and, since $\gamma \in (0,1]$, $$\|\langle \cdot\rangle^{\frac{\gamma}{2}}f(t)\|_{1} \leq \|\langle \cdot \rangle^{\gamma}f(t)\|_{1} \leq \|\langle \cdot \rangle^{\frac{3\gamma}{2}}f(t)\|_{1} \leq \|f(t)\|_{L^1_2}=\|f^{\rm in}\|_{L^1_2}\,,$$
since mass and energy are conserved during the evolution, 
we deduce that
\begin{multline} \label{eqproofL2-6}
\frac12 \frac{\d}{\d t} \|f(t)\|^2_{L^2_{\gamma}}+\mathbf{c}_{\gamma,f(t)} \, \|b\|_{1}  \,  \|\langle \cdot \rangle^{\frac{3\gamma}{2}}f(t)\|^2_{2} \leq 48 \|b^{\infty}\|_{\infty}  \left\|f^{\rm in}\right\|_{L^1_{2}}^{1 + \frac{2 \gamma}{3}}  \|\langle \cdot \rangle^{\frac{3\gamma}{2}}f(t)\|_{2}^{2 - \frac{2 \gamma}{3}} \\+ 4 \|b^1\|_{1} \|f^{\rm in}\|_{L^1_{2}} \|\langle \cdot \rangle^{\frac{3\gamma}{2}}f(t)\|^{2}_{2} \,.
\end{multline}
We now recall that $\mathbf{c}_{\gamma,f(t)}$, coming from Proposition~\ref{prop:lwrbdQ-} only depends on $f(t)$ through $\|f\|_{1}$, $\|f\|_{L^1_2}$ and an upper bound on $\|f\|_{L^1_3}$. Using now~\eqref{eq:supL13}, we conclude to the existence of a constant $\mathbf{\widetilde{c}}_{\gamma, f^{\rm in}} > 0$ that depends only on $\gamma$ and $f^{\rm in}$, which is \emph{explicit and independent of $\e$}, such that for any $t \geq 0$, we have $\mathbf{c}_{\gamma,f(t)} \geq \mathbf{\widetilde{c}}_{\gamma, f^{\rm in}}$. Therefore,~\eqref{eqproofL2-6} may be recast into
\begin{multline*}
\frac12 \frac{\d}{\d t} \|f(t)\|^2_{L^2_{\gamma}}+\mathbf{\widetilde{c}}_{\gamma, f^{\rm in}}\, \|b\|_{1}  \,  \|\langle \cdot \rangle^{\frac{3\gamma}{2}}f(t)\|^2_{2} \leq 48 \|b^{\infty}\|_{\infty}  \left\|f^{\rm in}\right\|_{L^1_{2}}^{1 + \frac{2 \gamma}{3}}  \|\langle \cdot \rangle^{\frac{3\gamma}{2}}f(t)\|_{2}^{2 - \frac{2 \gamma}{3}}\\ + 4 \|b^1\|_{1} \|f^{\rm in}\|_{L^1_{2}} \|\langle \cdot \rangle^{\frac{3\gamma}{2}}f(t)\|^{2}_{2} \,.\end{multline*}
We can choose $b^{\infty} \in L^{\infty}((-1,1))$ such that
$$
\|b^1\|_{1} = \|b-b^{\infty}\|_{1} \leq \frac{\mathbf{\widetilde{c}}_{\gamma,f^{\rm in}} \, \|b\|_{1}}{8 \|f^{\rm in}\|_{L^1_2}}.
$$
For such a choice of $b^\infty$, we set
$\frac12 A_1 = 48 \left\|f^{\rm in} \right\|_{L^1_{2}}^{1 + \frac{2 \gamma}{3}} \|b^{\infty}\|_{\infty}$ and  deduce that
\begin{multline*}
\frac12 \frac{\d}{\d t} \|f(t)\|^2_{L^2_{\gamma}} +\mathbf{\widetilde{c}}_{\gamma, f^{\rm in}}\, \|b\|_{1}  \,  \|\langle \cdot \rangle^{\frac{3\gamma}{2}}f(t)\|^2_{2} \leq \\ \frac12 A_1 \|\langle \cdot \rangle^{\frac{3\gamma}{2}}f(t)\|_{2}^{2 - \frac{2 \gamma}{3}} + \frac{1}{2} \mathbf{\widetilde{c}}_{\gamma,f^{\rm in}} \, \|b\|_{1} \,   \|\langle \cdot \rangle^{\frac{3 \gamma}{2}}f(t)\|^{2}_{2},
\end{multline*}
that is
\begin{equation} \label{eq:similarform}
\frac{\d}{\d t} \|f(t)\|^2_{L^2_{\gamma}} \leq A_1 \, \|\langle \cdot \rangle^{\frac{3 \gamma}{2}}f(t)\|_{2}^{2 - \frac{2 \gamma}{3}} - 2 A_2 \, \|\langle \cdot \rangle^{\frac{3 \gamma}{2}}f(t)\|_{2}^2,
\end{equation}
where we denoted $A_2 := \frac{1}{2} \mathbf{\widetilde{c}}_{\gamma,f^{\rm in}} \, \|b\|_{1}$. Notice that $A_1 \geq 0$ and $A_2 > 0$ \emph{do not depend on time, nor on $\varepsilon$}. A study of $X \mapsto A_1 X^{2 -  \frac{2 \gamma}{3}} -  A_2 X^2$ leads to, for any $ X \geq 0$,
$$
A_1 X^{2 - \frac{2 \gamma}{3}} - A_2 X^2 \leq \left(\frac{\frac{2 \gamma}{3}}{2^{\frac{3}{\gamma}}(2 - \frac{2 \gamma}{3})^{1-\frac{3}{\gamma}}} \right) \times \frac{ A_1^{\frac{3}{\gamma}}}{A_2^{\frac{3}{\gamma}-1}},
$$
The term inside the parentheses in the right-hand side is actually increasing in $\gamma$, thus can be bounded by its value at $\gamma=1$, and we conclude that for any $X \geq 0$,
$$
 A_1 X^{2 - \frac{2 \gamma}{3}} - A_2 X^2 \leq \frac{4}{27} \times \frac{ A_1^{\frac{3}{\gamma}}}{A_2^{\frac{3}{\gamma}-1}} =: A_3.
$$
Noticing that $\|\langle \cdot \rangle^{\frac{3\gamma}{2}}f(t)\|_2 \geq \|f(t)\|_{L^2_{\gamma}}$, this leads us to
$$
\frac{\d}{\d t} \|f(t)\|^2_{L^2_{\gamma}} \leq A_3 - A_2  \|\langle \cdot \rangle^{\frac{3\gamma}{2}}f(t)\|^2_{2} \leq A_3 - A_2  \|f(t)\|^2_{L^2_{\gamma}}, \qquad \forall \, t \geq 0.
$$
A straightforward ODE integration and the fact that $f(t=0) = f^{\rm in}$ then yields, for any $t \geq 0$,
$$
\|f(t)\|^2_{L^2_{\gamma}} \leq \max \left(\|f^{\rm in}\|^2_{L^2_{\gamma}}, \, \frac{A_3}{A_2} \right).
$$
We conclude by taking the supremum in time. Finally, note that $\|f^{\rm in}\|^2_{L^2_{\gamma}}$ is bounded by $\|f^{\rm in}\|_{L^1_2}\|f^{\rm in}\|_{\infty}$, and $\frac{A_3}{A_2}$ is \emph{explicit} and depends only on $\gamma$, $b$ and $f^{\rm in}$ through its $L^1$ and $L^1_2$ norms and an upper-bound on its $L^1_3$ norm - in particular, \emph{not on $\e$}.
\end{proof}

\subsection{\texorpdfstring{Conclusion to the proof of Theorem~\ref{theorem:linftybound}}{The Linfty bound}}
We assume that the assumptions of Theorem ~\ref{theorem:linftybound} are in force and adopt the notations introduced therein. We provide in this subsection the conclusion of the proof of this Theorem. As in the proof of Proposition~\ref{prop:L2bound}, we denote in the following $f(t) \equiv f^\e(t,\cdot)$ for any $t \geq 0$ to avoid ambiguity and for clarity, and we have for any $v \in \R^3$,
$$
\partial_t f(t)(v) \leq \Q^+_{\gamma,b}(f(t),f(t))(v) + \frac{f(t)(v)}{\|f(t)\|_{\infty}}\Gamma_{\gamma,b}(f(t),f(t))(v) - \Q^-_{\gamma,b}(f(t),f(t))(v), \qquad t\geq 0.
$$
Following the ideas of~\cite[Subsection 7.3]{alonso2022boltzmann}, we split $b = b_S + b_R$ with $b_S$ vanishing in the vicinity of $\{-1,1\}$. Since $\Q^+_{\gamma,b}$ and $\Gamma_{\gamma,b}$ depend \emph{linearly} on the kernel $b$, we observe that, for any $(t,v) \in \R_+ \times \R^3$,
\begin{equation}\label{eqproofLinf-1}
\begin{split}
  \partial_t f(t)(v) +\Q^-_{\gamma,b}(f(t),f(t))(v) &\leq \Q^+_{\gamma,b_S}(f(t),f(t))(v)  
  + \Q^+_{\gamma,b_R}(f(t),f(t))(v) \\ 
  &\phantom{++} +\frac{f(t)(v)}{\|f(t)\|_{\infty}}\Gamma_{\gamma,b_S}(f(t),f(t))(v)  
  \\
  &\phantom{+++} + \;\frac{f(t)(v)}{\|f(t)\|_{\infty}}\Gamma_{\gamma,b_R}(f(t),f(t))(v)\,.
\end{split}\end{equation}
Here again, all the above terms have been studied and estimated in  Subsection~\ref{subsect:estimates}. Namely, 
Equation~\eqref{eq:bsQ+} in Proposition~\ref{prop:bsbRQ+} allows to bound $\Q^+_{\gamma,b_S}$ from above, Equation~\eqref{eq:bRQ+} in Proposition~\ref{prop:bsbRQ+} allows to estimate $\Q^+_{\gamma,b_R}$ while~\eqref{eq:bsGamma} and~\eqref{eq:bRGamma} provide the estimates for $\Gamma_{\gamma,b_S}$ and $\Gamma_{\gamma,b_R}$ respectively. Finally, $\Q^{-}_{\gamma,b}$ is bounded from below thanks to~\eqref{eq:loweQQ}. With all these estimates,~\eqref{eqproofLinf-1}  becomes
\begin{equation*}\begin{split}
  \partial_t f(t)(v) +\mathbf{c}_{\gamma, f(t)}\,  \|b\|_{1} \, f(t)(v)\, \langle v \rangle^{\gamma}\leq \; \; \; \;&C_{b_S}(2,2) \, \|f(t)\|_{L^2_{\gamma}}^2 \\
+ \; &4 \, \|b_R\|_{1} \, \|f(t)\|_{1}\,\|f(t)\|_{\infty} \, \langle v \rangle^{\gamma} \\
+ \; &\left( C_{b_S}(1,2) + C_{b_S}(2,1) \right)  \langle v \rangle^{\gamma}  \; \|f(t)\|_{L^2_{\gamma}} \,\|f(t)\|_{2} \\
+ \; & 2 \, \|b_R\|_{1} \, \|f(t)\|_{L^1_{\gamma}} \, f(t)(v) \, \langle v \rangle^{\gamma}\,.\end{split}\end{equation*}
As already explained in the lines following~\eqref{eqproofL2-6}, there exists an explicit $\mathbf{\widetilde{c}}_{\gamma, f^{\rm in}} > 0$ depending only on $\gamma$ and $f^{\rm in}$ such that for any $t\geq 0$, $\mathbf{c}_{\gamma, f(t)} \geq \mathbf{\widetilde{c}}_{\gamma, f^{\rm in}}$. Moreover we have for any $t \geq 0$ that
$$
\|f(t)\|_{1} \leq \|f(t)\|_{L^1_{\gamma}} \leq \|f(t)\|_{L^1_2} = \|f^{\rm in}\|_{L^1_2}.
$$
Moreover remarking that $\|f(t)\|_{2} \leq \|f(t)\|_{L^2_{\gamma}}$, $f(t)(v) \leq \|f(t)\|_{\infty}$ and $1 \leq \langle v \rangle^{\gamma}$, we then obtain
\begin{multline*}
 \partial_t f(t,v) \leq \; \; \; \; \left(C_{b_S}(2,2) + C_{b_S}(1,2) + C_{b_S}(2,1) \right) \|f(t)\|_{L^2_{\gamma}}^2 \,  \langle v \rangle^{\gamma} \\
 + \; 6 \|b_R\|_{1} \, \|f^{\rm in}\|_{L^1_2} \, \|f(t)\|_{\infty} \, \langle v \rangle^{\gamma} 
 - \; \mathbf{\widetilde{c}}_{\gamma, f^{\rm in}} \, \|b\|_{1} \, f(t)(v)\, \langle v \rangle^{\gamma}.
\end{multline*}
We now apply Proposition~\ref{prop:L2bound}, providing the uniform-in-time $L^2_{\gamma}$ bound
$$
\sup_{t\geq0}\|f(t)\|_{L^2_{\gamma}}^2 \leq  \mathbf{C}_{2,\gamma}^2.
$$
We choose the kernel $b_S$ close enough (in the $L^1$ sense) to $b$, to ensure that
$$
\|b_R\|_{1} \leq \frac{\mathbf{\widetilde{c}}_{\gamma, f^{\rm in}} \, \|b\|_{1}}{12 \|f^{\rm in}\|_{L^1_2}}.
$$
In this event, we obtain that for any $(t,v) \in \R_+ \times \R^3$ we have
$$
\partial_t f(t,v) \leq A'_1 \, \langle v \rangle^{\gamma} + A_2 \, \|f(t)\|_{\infty} \, \langle v \rangle^{\gamma} - 2 A_2 \, f(t)(v) \, \langle v \rangle^{\gamma},
$$
where
$$
A'_1 = \left(C_{b_S}(2,2) + C_{b_S}(1,2) + C_{b_S}(2,1) \right) \mathbf{C}_{2,\gamma}^2 \qquad \text{and} \qquad A_2 = \frac12 \mathbf{\widetilde{c}}_{\gamma, f^{\rm in}} \, \|b\|_{1}
$$
are explicit and depend only on $\gamma$, $b$ and $f^{\rm in}$ through its $L^1$ and $L^1_2$ norms and upper-bounds on its $L^1_3$ and $L^{\infty}$ norms - in particular, neither on time $t$ nor on $\e$. Fixing $t > 0$, we then have for any $s \in [0,t]$ that, at fixed $v \in \R^3$,
$$
\partial_s f(s,v) + 2 A_2 \, f(s,v) \, \langle v \rangle^{\gamma} \leq (A'_1 + A_2 \sup_{0\leq t' \leq t} \|f(t')\|_{\infty})\langle v \rangle^{\gamma}.
$$
Multiplying the above inequation by $e^{-2 A_2 \, \langle v \rangle^{\gamma} s}$ and integrating for $s \in [0,t]$, while keeping $v$ fixed, yields, as the right-hand-side does not depend on $s$,
$$
f(t,v) \leq f^{\rm in}(v) \, e^{-2 A_2 \, \langle v \rangle^{\gamma} t} + \left(\frac{A'_1}{2 A_2} + \frac12 \sup_{0\leq t' \leq t} \|f(t')\|_{\infty} \right) (1 - e^{-2 A_2 \, \langle v \rangle^{\gamma} t}).
$$
In particular,
$$
f(t,v) \leq \max \left(\|f^{\rm in}\|_{\infty}, \; \frac{A'_1}{2 A_2} + \frac12 \sup_{0\leq t' \leq t} \|f(t')\|_{\infty}  \right).
$$
As a consequence
$$
\sup_{0\leq t' \leq t} \|f(t')\|_{\infty} \leq  \max \left(\|f^{\rm in}\|_{\infty}, \; \frac{A'_1}{2 A_2} + \frac12 \sup_{0\leq t' \leq t} \|f(t')\|_{\infty}  \right),
$$
and, since $\sup_{0\leq t' \leq t} \|f(t')\|_{\infty} \leq \e^{-1} < + \infty$, we easily deduce the bound
$$
\sup_{0\leq t' \leq t} \|f(t')\|_{\infty} \leq  \max \left(\|f^{\rm in}\|_{\infty}, \; \frac{A'_1}{A_2}  \right).
$$
Theorem~\ref{theorem:linftybound} follows as the right-hand side does not depend on time $t$, and $\frac{A'_1}{A_2}$ is explicit and depends only on $\gamma$, $b$ and $f^{\rm in}$ through its $L^1$ and $L^1_2$ norms and upper-bounds on its $L^1_3$ and $L^{\infty}$ norms - in particular, not on $\e$.

\section{\texorpdfstring{Explicit rate of convergence to equilibrium: proof of Theorem~\ref{theo:main}}{Explicit rate of convergence to equilibrium: proof of main Theorem}}\label{sec:maintheo}

To deduce from the above result the proof of Theorem~\ref{theo:main}, it remains to obtain some suitable conditions on the initial datum $f^{\rm in}$ ensuring that the solution $f^\varepsilon$ to~\eqref{eq:BFDequation} belongs to the class $\mathcal{C}_\varepsilon$ as described in the introduction. In particular, we first need to prove the existence of pointwise lower bounds of exponential type. This is the object of the next section which, once again, fully exploits known results for the classical Boltzmann equation as well as a comparison principle between $\Q^\varepsilon_B$ and $\Q_B,$

\subsection{Maxwellian pointwise lower bound} The following Proposition~\ref{prop:maxwellianlb} is the equivalent to~\cite[Theorem~1.1]{pulvirentimaxellian} for the fermionic case, the latter being stated for the classical Boltzmann equation.

\label{subsect:maxwellianlb}
\begin{prop} \label{prop:maxwellianlb}
Consider the assumptions of Theorem ~\ref{theorem:linftybound} along with the same notations. Then for any $\kappa_0 \in (0,1)$, $\e \in \left(0, (1-\kappa_0) \mathbf{C}_{\infty}^{-1} \right]$ and positive time $t_0 > 0$, there exist two positive constants $K_0$ and $A_0$ depending only on $t_0$, $\kappa_0$, $\gamma$, $b$ and $f^{\rm in}$ only through its $L^1$ and $L^1_2$ norms and an upper bound on its entropy, such that the solution $f^\e$ to~\eqref{eq:BFDequation} is such that, for all $t \geq t_0$,
\begin{equation} \label{eq:propMaxwelllb}
\qquad \qquad \qquad \qquad f^{\e}(t,v) \geq K_0 \, e^{-A_0 \, |v|^2}, \qquad \qquad v\in \R^3.
\end{equation}
\end{prop}
\begin{nb} We insist on the fact that, as in the previous results, the positive constants $K_0,A_0$ in~\eqref{eq:propMaxwelllb} \emph{do not depend on $\varepsilon$}.\end{nb}
\begin{proof} We do not provide a full proof of this result since it is, actually, exactly the same proof as the one valid for the classical Boltzmann equation. Indeed, the key point here is that the solution $f^\varepsilon$ satisfies
\begin{equation} \label{eq:lowerbfdkapppa0}
\partial_t f^{\e}(t,v) \geq \kappa_0^2 \,  \Q^{+}_{\gamma,b}(f^{\e}(t),f^{\e}(t))(v) -  \Q^{-}_{\gamma,b}(f^{\e}(t),f^{\e}(t))(v) \qquad \forall (t,v) \in \R_+\times \R^3\,.
\end{equation}
This is easily seen recalling the splittings~\eqref{eq:Qsplit}--\eqref{eq:Qsplite} and observing that, under the assumption $\kappa_0 \leq 1-\e f \leq1$, it holds
$$\Q^{\e,+}_{\g,b}(f,f) \geq \kappa_0^2 \Q^+_{\gamma,b}(f,f), \qquad \Q^{\e,-}_{\g,b}(f,f) \leq \Q^-_{\gamma,b}(f,f).$$
Since $\partial_t f^\e=\Q^{\e,+}_{\g,b}(f,f)-\Q^{\e,-}_{\g,b}(f,f)$, we deduce readily from Corollary~\ref{corlinfkappa0}  that $f^\e$ satisfies~\eqref{eq:lowerbfdkapppa0}. With this, one sees that $f^\e$ is a ``supersolution'' (in the sense of~\eqref{eq:lowerbfdkapppa0}) of an equation very similar to the classical Boltzmann equation, the only difference being that the $\Q^+_{\g,b}$ operator is multiplied by $\kappa_0^2$. This allows to resume the whole proof of~\cite[Theorem~1.1]{pulvirentimaxellian} (the fact that we are dealing with ``supersolution'' and not solution to~\eqref{eq:lowerbfdkapppa0} plays no role since we consider lower bounds for $f^\e$). We  notice that, although the conservation of mass and energy and the decrease of the entropy are not embedded in~\eqref{eq:lowerbfdkapppa0}, they do hold as $f^{\e}$ solves~\eqref{eq:BFDequation} and this allows to copycat the  proof of~\cite[Theorem~1.1]{pulvirentimaxellian}.\end{proof}

\subsection{Conclusion to the proof of Theorem~\ref{theo:main}}\label{sec:concl} In this last Subsection, we provide the core of the proof of Theorem~\ref{theo:main} and conclude. We recall that we assume here that $B$ is a collision kernel of the form~\eqref{eq:HypBB} with $\gamma \in (0,1]$ and an angular kernel $b$ satisfying the cutoff assumption~\eqref{cutoffassumption}. We fix an initial datum $0 \leq f^{\rm in} \in L^1_3(\R^3) \cap L^{\infty}(\R^3)$. 

According to Theorem~\ref{theorem:linftybound} (more precisely Corollary~\ref{corlinfkappa0}), there is an explicit $\mathbf{C}_{\infty} > 0$ such that for any $\kappa_0 \in (0,1)$ and $\e \in (0, (1-\kappa_0) \, \mathbf{C}_{\infty}^{-1}]$, the associated solution $f^\e$  to~\eqref{eq:BFDequation} with $f^\e(t=0)=f^{\rm in}$ is such that 
\begin{equation} \label{eqprooffinal:Linfty}
1 - \e f^{\e}(t,v) \geq \kappa_0 \qquad (t,v) \in \R_+\times\R^3.
\end{equation}
From now on, we fix $\kappa_0$ and $\e$ such that~\eqref{eqprooffinal:Linfty} holds true and  denote for simplicity $f(t) \equiv f^{\e}(t,\cdot)$.

\medskip

\noindent We arbitrarily choose a positive time $t_0 > 0$ (one could for instance take $t_0 = 1$). Using the result of~\cite[Theorem 2. (I)]{lu2001spatially}, we deduce that, for any $s > 2$, there exists an explicit constant $\mathbf{C}_{1,s} > 0$ (depending on $s,B,\|f^{\rm in}\|_{1},\|f^{\rm in}\|_{L^1_2}$ and $t_0$ but not $\e$) such that
\begin{equation} \label{eqprooffinal:L1sbd}
\sup_{t \geq t_0} \|f(t)\|_{L^1_s} \leq \mathbf{C}_{1,s}.
\end{equation}
Moreover, a simple interpolation together with  
Theorem~\ref{theorem:linftybound} yields
\begin{equation} \label{eq:supfLpfirst}
\sup_{t \geq 0}\|f(t)\|_p \leq \|f^{\rm in}\|_1^{\frac{1}{p}}\,\mathbf{C}_\infty^{1-\frac{1}{p}} \qquad \forall p >1.
\end{equation}
Moreover, we recall the notation (it is not  a norm), for any measurable $h : \R^3 \mapsto \R_+$ and $s \geq 0$,
$$
\|h\|_{L^1_{s} \log  L} = \int_{\R^3} \langle v \rangle^s \, h(v) \, |\log h(v)| \, \d v.
$$
When $h \in L^1_{s + 1}(\R^3) \cap L^\infty (\R^3)$, the above can be bounded as follows,
$$
\|h\|_{L^1_{s} \log  L} \leq \log^+ \left(\|h\|_{\infty} \right) \, \|h\|_{L^1_s} + \|h\|_{L^1_{s+1}} + \int_{\R^3} e^{-\langle v \rangle} \langle v \rangle^{s+1} \, \d v.
$$
Therefore, combining the above (with $h = f(t)$) with~\eqref{eqprooffinal:L1sbd} and the $L^\infty$ bound on $f$, it holds that
\begin{equation} \label{eq:boundL1slogL}
\sup_{t \geq t_0} \, \|f(t)\|_{L^1_{s} \log  L} \leq \log^+ \left(\mathbf{C}_{\infty} \right)  \mathbf{C}_{1, \, s} +  \mathbf{C}_{1, s+1} +  \int_{\R^3} e^{-\langle v \rangle} \langle v \rangle^{s+1} \, \d v.
\end{equation}
Equations~\eqref{eq:supfLpfirst} and~\eqref{eq:boundL1slogL} can be reformulated as the fact that, for all $p>1$ and $s \geq 0$, there exist explicit $\mathbf{C}_p > 0$ and $\mathbf{C}_{s}^{\log} > 0$ such that
\begin{equation} \label{eqprooffinal:LpLlogL}
\sup_{t \geq t_0} \|f(t)\|_{p} \leq \mathbf{C}_{p}, \qquad \text{and} \qquad \sup_{t \geq t_0} \|f(t)\|_{L^1_s \log L} \leq \mathbf{C}^{\log}_{s}.
\end{equation}
In particular, both $\mathbf{C}_{p}$ and $\mathbf{C}^{\log}_{s}$ only depend on $\gamma$, $b$ and $f^{\rm in}$ through its $L^1$ and $L^1_2$ norms and upper-bounds on its $L^1_3$ and $L^\infty$ norms, and respectively on $p$ and $s$. Moreover, from Proposition~\ref{prop:maxwellianlb}, there exist $K_0 > 0$ and $A_0 > 0$ depending only on $t_0$, $\kappa_0$, $\gamma$, $b$ and $f^{\rm in}$ through its $L^1$ and $L^1_2$ norms and an upper bound on its entropy, such that for any $(t,v) \in [t_0,+\infty) \times \R^3$, we have
\begin{equation} \label{eqprooffinal:maxwel}
f(t,v) \geq K_0 \, e^{-A_0 |v|^2}.
\end{equation}
Let us now define, for any $t\geq0$ 
$$
g(t) :=\varphi_\e(f(t))= \frac{f(t)}{1 - \e f(t)}.
$$
Since, for any $t \in \R_+$, \;  $0 \leq f(t) \in L^1_2(\R^3)$ and $1 - \e f \geq \kappa_0$, we also have $0 \leq g(t) \in L^1_2(\R^3)$, as well as, from~\eqref{eqprooffinal:Linfty}--\eqref{eqprooffinal:LpLlogL}, that for any $p > 1$ and $s \geq 0$, it holds that

\begin{equation} \label{eqprooffinal:bdg}
\sup_{t \geq t_0} \|g(t)\|_{p} \leq \kappa_0^{-1} \, \mathbf{C}_{p}, \qquad \text{and} \qquad \sup_{t \geq t_0} \|g(t)\|_{L^1_s \log L} \leq \kappa_0^{-1} \,  \mathbf{C}^{\log}_{s} + \kappa_0^{-1} \log (\kappa_0^{-1}) \, \mathbf{C}_{1,s}.
\end{equation}
Moreover, since $g \geq f$, Equation~\eqref{eqprooffinal:maxwel} implies
$$
g(t,v) \geq K_0 \, e^{-A_0 |v|^2}, \qquad \forall t \geq t_0, \quad v \in \R^3.
$$

\medskip

\noindent We are now able to apply the functional inequality in Theorem~\ref{theo:Vil} to the function $g(t)$, for $t \geq t_0$, which, with the choices
$$
\beta_- = 0, \qquad \beta_+ = \gamma, \qquad p = 3, \qquad s = 4 + \frac{4}{\gamma}, \qquad \text{and} \qquad q_0 = 2 \quad \text{ so that } \alpha=1+\gamma,
$$
 provides the existence of a positive constant  $\bm{A}_0 > 0$  such that
 \begin{equation} \label{eqprooffinal:classicalentropy}
    \mathscr{D}_0(g(t)) \geq \bm{A}_0  \, \mathcal{H}_0\left(g(t) \big|\M_0^{g(t)}\right)^{1 + \gamma}.
\end{equation}
We point out that the positive constant $\bm{A}_0$ can be defined as $\bm{A}_0=\inf_{t \geq t_0}\mathbf{A}(t)$
where, according to Theorem~\ref{theo:Vil}, $\mathbf{A}(t)$ is depending on $A_0,K_0$, upper and lower bounds to $\|g(t)\|_1$ and $\|g(t)\|_{L^1_2}$, and upper bounds to $\|g(t)\|_3,\|g(t)\|_{L^1_{s+2}}$ and $\|g(t)\|_{L^1_s\log L}$. This shows in particular that the positive constant $\bm{A}_0$ depends  only on $t_0$, $\varrho^{\rm in}$, $u^{\rm in}$, $E^{\rm in}$, $\kappa_0$, $K_0$, $A_0$, $\mathbf{C}_3,\mathbf{C}_s^{\log}$ and $\mathbf{C}_{1,s+2}.$

\medskip

\noindent We now apply Proposition~\ref{prop:TB} to get that
$$\mathcal{H}_{\varepsilon}(f(t)\big|\M_{\varepsilon}^{f(t)}) \leq \mathcal{H}_{0}\left(g(t)\big| \M_0^{g(t)}\right) \qquad \text{ and } \quad  \mathscr{D}_{\varepsilon}(f(t)) \geq \kappa_0^4 \, \mathscr{D}_0\left(g(t)\right),$$
which, combined with~\eqref{eqprooffinal:classicalentropy} yields
\begin{equation} \label{eqprooffinal:epsilonentropy}
    \mathscr{D}_{\varepsilon}(f(t)) \geq \kappa_0^4 \,  \bm{A}_0 \, \mathcal{H}_{\varepsilon}(f(t)\big|\M_{\varepsilon}^{f(t)})^{1 + \gamma}.
\end{equation}
Now notice that for any $t \geq 0$, we have, since $f$ solves \bfd \!\!,
$$
\M_{\varepsilon}^{f(t)} = \M_{\varepsilon}^{f^{\rm in}}.
$$
According to the entropy identity~\eqref{entropyidentity} (see Theorem~\ref{theo:cauchy}),  for all $t\geq 0$, we have
$$
\frac{\d}{\d t} \mathcal{H}_{\varepsilon}(f(t)\big|\M_{\varepsilon}^{f^{\rm in}}) = \frac{\d}{\d t} \mathcal{H}_{\varepsilon}(f(t)) = - \mathscr{D}_{\varepsilon}(f(t)).
$$
In particular, according to~\eqref{eqprooffinal:epsilonentropy},
$$\frac{\d}{\d t} \mathcal{H}_{\varepsilon}(f(t)\big|\M_{\varepsilon}^{f^{\rm in}}) \leq -\kappa_0^4 \,  \bm{A}_0 \, \mathcal{H}_{\varepsilon}(f(t)\big|\M_{\varepsilon}^{f^{\rm in}})^{1 + \gamma}, \qquad \forall t \geq t_0.$$
Integrating this inequality  on the interval $[t_0,t]$ and using, as the entropy decreases with time, that $$\mathcal{H}_\e\left(f(t_0)\big|\M_\e^{f^{\rm in}}\right) \leq \mathcal{H}_\e\left(f^{\rm in}\big|\M_\e^{f^{\rm in}}\right),$$ 
we easily deduce that
$$
\mathcal{H}_{\varepsilon}(f(t)\big|\M_{\varepsilon}^{f^{\rm in}}) \leq \left(\mathcal{H}_{\varepsilon}(f^{\rm in}\big|\M_{\varepsilon}^{f^{\rm in}})^{-\gamma} + \gamma \kappa_0^{4} \, \bm{A}_0\, (t-t_0) \right)^{-\frac{1}{\gamma}}, \qquad \forall t \geq t_0.
$$
Finally, using again that the mapping $t  \mapsto \mathcal{H}_{\varepsilon}(f(t)\big|\M_{\varepsilon}^{f^{\rm in}})$ is nonincreasing, we deduce that for all $t \geq 0$, it holds that
$$
\mathcal{H}_{\varepsilon}(f(t)\big|\M_{\varepsilon}^{f^{\rm in}}) \leq \mathbf{B} \, \left(1 + (t-t_0)_+ \right)^{-\frac{1}{\gamma}},
$$
with
$$
\mathbf{B} = \max \left(\mathcal{H}_{\varepsilon}(f^{\rm in}\big|\M_{\varepsilon}^{f^{\rm in}}), \left( \gamma \kappa_0^{ 4} \, \bm{A}_{0}  \right)^{-\frac{
1}{\gamma}} \right).
$$
The latter then implies~\eqref{eqtheorem:explicitcvH} for all $\e \in (0,(1-\kappa_0) \, \mathbf{C}_{\infty}^{-1}]$, with 
$\mathbf{C}_{\HH} = (1 + t_0)^{\frac{1}{\gamma}} \, \mathbf{B}.$  

\medskip

\noindent While for~\eqref{eqtheorem:explicitcvH}, one could take $\e^{\rm in} = (1-\kappa_0) \, \mathbf{C}_{\infty}^{-1}$, we choose for the following
$$
\e^{\rm in} = \min \left((1-\kappa_0) \, \mathbf{C}_{\infty}^{-1}, \; \e^{\dag}_{\rm sat} \right),
$$
where $\e^{\dag}_{\rm sat}$ is defined in Lemma~\ref{lemma:appendixFDS} in Appendix~\ref{appendix:FDS}, is explicit and (as we apply the lemma to $f = f^{\rm in}$) depends only on $\varrho^{\rm in}$ an $E^{\rm in}$ (it is approximately $0.06 \, \e_{\rm sat}$).

\medskip

We finally show~\eqref{eqtheorem:explicitcvLpk} for all $\e \in (0,\e^{\rm in}]$. We make use of the weighted $L^p$ Cszisàr-Kullback-Pinsker inequality recalled in Proposition~\ref{prop:CKP} in Appendix~\ref{app:cauchy}, with $p=1$ and $\varpi = \langle \cdot \rangle^k$ for some $k \geq 0$, giving for any $t \geq 0$, as  $\M_{\e}^{f(t)}=\M_{\e}^{f^{\rm in}}$,
$$
\left\|f(t) - \M_{\e}^{f^{\rm in}} \right\|_{L^1_k}^2 \leq 2\max\left(\left\|\M_{\e}^{f^{\rm in}} \right\|_{L^1_{2k}}\,,\,\|f(t)\|_{L^1_{2k}}\right)\mathcal{H}_\e\left(f(t)\big|\M_{\e}^{f^{\rm in}}\right).
$$
Since $\e \leq \e^{\dag}_{\rm sat}$, Lemma~\ref{lemma:appendixFDS} in Appendix~\ref{appendix:FDS} provides the existence of $\mathbf{C}_{\M, 1,2k}$, explicit and depending only on $\varrho^{\rm in}$ and $E^{\rm in}$ (in particular, not on $\e$) such that $\|\M_{\e}^{f^{\rm in}}\|_{L^1_{2k}} \leq \mathbf{C}_{\M, 1, 2k}$. Then, letting
\begin{equation}
\mathbf{C}_{\HH,1,k} = \sqrt{2\max\left(\mathbf{C}_{\M, 1,2k} \,,\,\mathbf{C}_{1,2k}\right) \mathbf{C}_{\HH}},
\end{equation}
we obtain, using the proven~\eqref{eqtheorem:explicitcvH},
\begin{equation} \label{eq:finalproofeqL1k}
\left\|f(t) - \M_{\e}^{f^{\rm in}} \right\|_{L^1_k} \leq \mathbf{C}_{\HH,1,k} \, (1 + t)^{-\frac{1}{2\gamma}}.
\end{equation}
Finally, let $p > 1$. As $$\|\cdot\|_{L^p_k} \leq \|\cdot\|^{1 - \frac{1}{p}}_{\infty} \, \|\cdot\|_{L^1_{pk}}^{\frac{1}{p}},$$
and since, as $\e \leq \e^{\dag}_{\rm sat}$, Lemma~\ref{lemma:appendixFDS} in Appendix~\ref{appendix:FDS} provides the existence of $\mathbf{C}_{\M, \infty}$, explicit and depending only on $\varrho^{\rm in}$ and $E^{\rm in}$ (in particular, not on $\e$) such that $\|\M_{\e}^{f^{\rm in}}\|_{\infty} \leq \mathbf{C}_{\M, \infty}$, we deduce~\eqref{eqtheorem:explicitcvLpk} from~\eqref{eq:finalproofeqL1k}, with
$$
\mathbf{C}_{\HH,p,k} =  \max \left( \mathbf{C}_{\M, \infty}, \; \mathbf{C}_{\infty} \right)^{1 - \frac{1}{p}} \, {\mathbf{C}_{\HH,1,pk}}^{\frac{1}{p}},
$$
where we used the fact that $\|f(t) - \M_{\e}^{f^{\rm in}} \|_{\infty} \leq \max \left(\|\M_{\e}^{f^{\rm in}} \|_{\infty}, \left\|f(t) \right\|_{\infty} \right)$, which holds as both $f(t)$ and $\M_{\e}^{f^{\rm in}}$ are nonnegative.

\medskip

\noindent Since every presented constant is explicit and depends only (at most) on $\gamma$, $b$, $\varrho^{\rm in}$, $u^{\rm in}$, $E^{\rm in}$ and upper-bounds to $\|f^{\rm in}\|_{L^1_3}$ and $\|f^{\rm in}\|_\infty$, in particular not on $\e$, the proof of Theorem~\ref{theo:main} is complete.

\appendix

\section{More physically relevant models}\label{app:COKer}

We briefly discuss here the possibility to recover the results established in the core of the text when dealing with more realistic collision kernels.
We begin with briefly recalling some facts about such physically relevant kernels.
\subsection{Quantum collision kernels} 
While collision kernels $B$ are fully explicit for classical particles, the situation is much involved for quantum (pseudo)-particles. In this case, \bfd  has been derived from the Schrödinger equation in the weak-coupling regime and the derived kernel $B$ takes the form
\begin{multline}\label{eq:phiB}
B(v,\vet,\sigma)=|z|\left[\widehat{\phi}\left(\left|z\sin\left(\frac{\theta}{2} \right)\right|\right) - \widehat{\phi}\left(\left|z\cos\left(\frac{\theta}{2}\right)\right|\right)\right]^2,\\ \qquad z=v-\vet, \; \cos\theta=\frac{z}{|z|}\cdot \sigma\,,\end{multline}
and where $\widehat{\phi}$ is the (generalized) Fourier transform of the  particle interaction potential $\phi=\phi(|x|)$, $x\in\R^3.$
As in~\cite{liulu}, we make the general assumption on $B$:
\begin{hyp}\label{hyp:Bgen}
The collision kernel $B(v,\vet,\sigma)=B(|v-\vet|, \sigma) = B(|v-\vet|, \cos \theta)$ is assumed to be such that 
$$|z|^\g \, \Phi_\ast(|z|) \, b_\ast(\cos\theta) \leq B(|z|,\cos\theta)=B(v,\vet,\sigma) \leq \left(1+|z|\right)^{\gamma}b^\ast(\cos\theta)
$$
for $\g \in (0,1]$ and some  Borel even functions $b_\ast(\cdot),b^\ast(\cdot)$ defined on $(-1,1)$ and a Borel function $\Phi_\ast\::\:\R_+\to\R_+$ such that
\begin{equation}\label{eq:b*b*}
 0 < b_\ast(\cos\theta) \leq b^\ast(\cos\theta), \quad \theta \in (-\pi,\pi); \qquad \int_0^{\pi}b^\ast(\cos\theta)\sin\theta\d\theta <\infty\end{equation}
 and
\begin{equation}\label{eq:Phir}\Phi_\ast(r) >0 \qquad \forall r >0, \qquad 
    \inf_{r \geq 1}\Phi_\ast(r) \geq 1, \quad \sup_{r \geq0}\Phi_\ast (r)<\infty\,.\end{equation}
    \end{hyp}
\begin{exa} It has been observed in \cite[Appendix A]{liulu} that, for potential interactions of the type
$$\phi(x)=|x|^{-\alpha}, \qquad 0< \alpha <3$$
the kernel $B$ defined by~\eqref{eq:phiB} is such that
$$B(z,\sigma)=|z|^{\g}b(\cos\theta), \qquad \g=2\alpha-5$$
with
$$b(\cos\theta)=C_\alpha\left((1-\cos\theta)^{-\beta}-(1+\cos\theta)^\beta\right)^2, \qquad \beta=\frac{3-\alpha}{2}$$
for some positive constant $C_\alpha$. One can check that  $b$ actually meets the cutoff assumption~\eqref{cutoffassumption}. In particular, it satisfies Assumption~\ref{hyp:Bgen} as soon as $\alpha \in \left(\frac{5}{2},3\right).$ Notice that such a kernel is of the form of kernels studied in the core of our work here.\end{exa}
\begin{exa} Choosing a potential of the form
$$\phi(|x|)=\frac{1}{2^\beta\Gamma(\beta)}\int_0^\infty G_t(|x|) \, t^{\beta-1} \, \exp\left(-\frac{t}{2}\right)\d t, \qquad x \in \R^3, \qquad 0 \leq \beta < \frac{1}{4}$$
where $\Gamma(\cdot)$ is the Gamma function and $G_t(|x|)$ is the heat kernel
$$G_t(|x|)=(2\pi t)^{-\frac{3}{2}}\exp\left(-\frac{|x|^2}{2t}\right), \qquad x \in \R^3, t >0$$
one can check that $B$ defined by~\eqref{eq:phiB} is such that
$$B(z,\sigma)=\frac{4^\beta|z|}{(2+|z|^2)^{2\beta}}\,\left(\left(1-a(|z|)\cos\theta\right)^{-\beta}-\left(1+a(|z|)\cos\theta\right)^\beta\right)^2, \qquad  a(|z|)=\frac{|z|^2}{2+|z|^2}.$$
It has been then observed in~\cite{liulu} that $B$ satisfies then Assumptions~\ref{hyp:Bgen} with $\g=1-4\beta \in (0,1]$,
$$\Phi_\ast(r)=\left(\frac{2 r^2}{1+r^2}\right)^{2\beta+2}, \qquad  b_\ast(\cos\theta)=c_\beta \cos^2\theta, \qquad b^\ast(\cos\theta)=C_\beta\cos^2\theta\,\sin^{-4\beta}(\theta)\,$$
for some positive constants $C_\beta,c_\beta >0.$ 
\end{exa}
\begin{exa}\label{exa:WR} A general subclass of kernels $B$ satisfying Assumptions~\ref{hyp:Bgen} has been considered in~\cite{WangRen} and corresponds to the choice $\g=1$ and
$$\Phi_\ast(r)=  \, {2} \, \frac{r^{ \beta}}{1+r^\beta}, \qquad   {\beta \geq 0},$$
and 
$$b_\ast(\cos\theta) \geq \, {\frac12} \, b_0 >0, \qquad b^\ast(\cos\theta) \leq b_1 < \infty.$$
This means in particular that
$$b_0\frac{|v-\vet|^{\beta+1}}{1+|v-\vet|^\beta} \leq B(v,\vet,\sigma) \leq b_1|v-\vet|.$$
\end{exa}\subsection{Main mathematical changes induced by quantum collision kernels}

For simplicity of presentation, let us assume that the collision kernel is of a type generalising the above example and assume that there exist $b_1 \geq b_0 >0$ and $\g,\beta \in (0,1)$ with $\g+\beta \in (0,1)$ such that
\begin{equation}\label{eq:WR}
b_0\frac{|v-\vet|^{\beta+\g}}{1+|v-\vet|^\beta} \leq B(v,\vet,\sigma) \leq b_1|v-\vet|^\g, \qquad v,\vet,\sigma \in \R^3\times\R^3\times\Sb^2.\end{equation}
We briefly explain here what should be the main changes/obstacles for the derivation of the results obtained in the core of the paper for kernels of the type \eqref{eq:HypBB}. 

Notice that, for such collision kernels, as pointed out already in~\cite{liulu}, the construction and properties of solutions as described in the above Appendix~\ref{app:cauchy} are easy to adapt (see the discussion hereafter).

We can check without too much difficulty that the whole set of results in Section~\ref{sec:LpLinf} are still valid under the above assumption \eqref{eq:WR} culminating in the following version of Theorem~\ref{theorem:linftybound}
\begin{theo}\label{theorem:linftyboundWR}
Let $\gamma \in (0,1]$ and $\beta\geq 0$ and let $B$ be a collision kernel satisfying \eqref{eq:WR} with $b_1\geq b_0 >0.$ For any $0 \leq f^{\rm in} \in L^1_3(\R^3) \cap L^{\infty}(\R^3)$, there exists an explicit $\mathbf{C}_{\infty}(B) > 0$, depending only on $B$ and $f^{\rm in}$ only through its $L^1$ and $L^1_2$ norms and upper-bounds on its $L^\infty$ and $L^1_3$ norms, such that for any $\e \in (0, \|f^{\rm in}\|_{\infty}^{-1}]$, the unique solution $f^\e$ to~\eqref{eq:BFDequation} associated to~$\e$, the collision kernel defined by~\eqref{eq:HypBB}  and initial datum $f^{\rm in}$ satisfies 
$$  \sup_{t \geq 0} \|f^{\e}(t)\|_{\infty} \leq \mathbf{C}_{\infty}(B).$$\end{theo}
 Indeed, the representation formula for $\Q^+_B$ allows again to define $\Gamma_B$ as in \eqref{eqdef:Q0+Q0bar} 
and, as in Section~\ref{sec:LpLinf},  we can define \begin{equation*} 
\widetilde{\Q}^{+}_{B}[f](g,h) := \Q^+_{B}(g,h) \, + \, \frac{f}{\|f\|_{\infty}} \, \Gamma_{B}(g,h),
\end{equation*}
and $\widetilde{\Q}_{B}[f](g,h):=\widetilde{\Q}_{B}^{+}[f](g,h)-\Q_{B}^{-}(g,h).$ With this, Proposition~\ref{prop:almostBoltz} and Lemma~\ref{lemma:adjointproperty} still hold with obvious change of notations. Since the results  in Section~\ref{sec:LpLinf} are obtained through estimates involving somehow the weak form of $\Q^+_B$, one checks easily that Proposition~\ref{prop:gainintegrQ+L2} (with $\|b\|_\infty$ replaced with $b_1$) can be deduced from Prop.~\ref{prop:youngineq}. All the results, up to Prop.~\ref{prop:bRGamma} remain then valid. The only change to be made lies in the proof of Proposition~\ref{prop:lwrbdQ-} but the proof of \cite[Lemma 4]{alonso2017} can be adapted to deduce the same result. All these results would yield Theorem~\ref{theorem:linftyboundWR}. 

Of course, the results recalled in the Introduction regarding the entropy production and entropy estimates remain valid for this class of collision kernel. The only result which does not seem to be directly deduced from existing results is the one in Section~\ref{subsect:maxwellianlb}. Typically, it would be very interesting to check whether Proposition~\ref{prop:maxwellianlb} (or a variant of it) still holds for collision kernels satisfying \eqref{eq:WR}. Notice that the obstacle has nothing to do here with the quantum nature of the Boltzmann operator and one should rather check if the classical Boltzmann operator $\Q^+_B$ is satisfying the estimates derived in  the original proof of~\cite{pulvirentimaxellian}. It seems to us that the key point to be checked is the spreading properties of the collision operator $\Q^+_B$ when $B(v,\vet,\sigma)=b_0\frac{|v-\vet|^{\beta+\g}}{1+|v-\vet|^{\beta}}$, namely, can we still prove that there exist $\eta >0,r >1$ such that
$$\Q^+_B(\mathds{1}_{\B(v_0,\delta_0)},\mathds{1}_{\B(v_0,\delta_0)}) \geq \eta \, \mathds{1}_{\B(v_0, \, r\delta_0)}$$
holds true for any $\delta_0 >0$, $v_0 \in \R^3$ ? Here above, $\B(v,\delta)$ denotes the closed ball of $\R^3$ centered at $v\in \R^3$ with radius $\delta >0.$ Such a result, obtained usually through the Carleman representation of the gain part operator, allows to initiate the iterative procedure yielding the derivation of pointwise bounds. To keep the paper simple enough, we did not elaborate on this point but are confident that Proposition~\ref{prop:maxwellianlb} is still true for general kernels $B$ of the form \eqref{eq:WR}. If this were the case,  {which is highly expected as, on a finite ball, the kernel $B$ behaves like $|v-v_*|^{\gamma + \beta}$, for which we know the result holds,} it would be straightforward then to resume the proof derived in Section~\ref{sec:concl} and obtain an analogue of Theorem~\ref{theo:main}.

\section{Known results about \bfd  and the Fermi-Dirac entropy}\label{app:cauchy}

We collect in this section some knwon facts about \bfd  obtained in~\cite{lu2001spatially,luwennberg} as well as some recent results regarding the relative Fermi-Dirac entropy obtained by the first author.

\subsection{Cauchy problem and moment estimates}

We briefly recall here the notion of solutions we consider in the present paper, the Cauchy result established in~\cite{lu2001spatially} as well as moments estimates.

We adopt the following framework and recall here that we always assume $B$ to be a collision kernel satisfying~\eqref{eq:HypBB}.
\begin{defi}\label{defi:Mild} Let $\varepsilon >0$ and $0 \leq f^{\rm in} \in L^1_2(\R^3)$ satisfying $1-\varepsilon f^{\rm in} \geq 0$. We say that a Lebesgue measurable function 
$$f^{\varepsilon}\::\:[0,\infty) \times \R^3 \to \left[0,\varepsilon^{-1}\right)$$ such that $\sup_{t\geq 0}\|f(t)\|_{L^1_2} < \infty$ is a weak solution to \bfd  if there is $\Lambda \in \R^3$ with zero Lebesgue measure such that 
$$\int_0^T \d t\int_{\R^3 \times \Sb^2}B(v,\vet,\sigma)G^{\pm}(t,v,\vet,\sigma)\d \vet\d\sigma < \infty, \qquad \forall 0 < T < \infty, \qquad \forall v \in \R^3 \setminus \Lambda$$
where
$$
G^+(t,v,\vet,\sigma)=\left[f' f'_* (1 - \varepsilon f) (1 - \varepsilon f_*)\right], 
\quad G^-(t,v,\vet,\sigma)=\left[f f_* (1 - \varepsilon f') (1 - \varepsilon f'_*)\right]
$$
and
$$f^{\varepsilon}(t,v)=f^{\rm in}(v)+\int_0^t \Q^{\varepsilon}_B(f^{\varepsilon},f^{\varepsilon})(\tau,v)\d \tau, \qquad t \geq0.$$
\end{defi}
For that class of solutions, existence of solutions have been established in \cite[Theorem 1]{dolbeaultFD} and stability, uniqueness of solutions as well as the entropy identity have been established in~\cite{lu2001spatially,luwennberg}:
\begin{theo}\label{theo:cauchy}  Given $\varepsilon >0$ and $0 \leq f^{\rm in} \in L^1_2(\R^3)$ satisfying $1-\varepsilon f^{\rm in} \geq 0$, there exists a unique solution $f^{\varepsilon}=f^{\varepsilon}(t,v)$ to~\eqref{eq:BFDequation} in the sense of Definition~\ref{defi:Mild}. Moreover, such a solution satisfies the entropy identity
\begin{equation}\label{entropyidentity}
\HH_{\varepsilon}(f^{\varepsilon}(t))=\HH_{\varepsilon}(f^{\rm in})- \int_0^t \mathscr{D}_{\varepsilon}\left(f^{\varepsilon}(\tau)\right)\d\tau, \qquad \forall t \geq 0.\end{equation}    
\end{theo}

We also recall that moments are created for solutions to~\eqref{eq:BFDequation} associated to a collision kernel of the form~\eqref{eq:HypBB} where we recall that $\gamma \in (0,1]:$
\begin{prop} Given $\varepsilon >0$ and $0 \leq f^{\rm in} \in L^1_2(\R^3)$ satisfying $1-\varepsilon f^{\rm in} \geq 0$, let  $f^{\varepsilon}=f^{\varepsilon}(t,v)$ be the unique conservative solution to~\eqref{eq:BFDequation} in the sense of Definition~\ref{defi:Mild}. For any $s >2$, one has
$${\|f^\varepsilon(t)\|_{L^1_s}}  \leq \left[\frac{a_1}{1-\exp\left(-a_2 t\right)}\right]^{\frac{s-2}{\gamma}}, \qquad \forall t > 0$$
where $a _1> 0, a_2 > 0$ are constants depending only on $B$, $\|f^{\rm in}\|_1$, $\|f^{\rm in}\|_{L^1_2}$ and not on $\varepsilon.$ Moreover, given $s > 2$, there exists $C_B(f^{\rm in}) >0$ depending only on $B$, $\|f^{\rm in}\|_1$, $\|f^{\rm in}\|_{L^1_2}$ and $\|f^{\rm in}\|_{L^1_s}$ but not on $\varepsilon$ such that
\begin{equation}\label{eq:msf}
    \sup_{t \geq0} {\|f^\varepsilon(t)\|_{L^1_s}} \leq C_B(f^{\rm in}).
\end{equation} 
\end{prop}
\begin{nb} Such a result is a simple consequence of ~\cite[Theorem 2. (I)]{lu2001spatially} where the \emph{creation} of an $L^1_s$ bound has been derived. We point out that,if one assumes the initial datum $f^{\rm in} \in L^1_s(\R^3)$, then the uniform-in-time bound provided in~\eqref{eq:msf} can be easily deduced from the equation \emph{just preceeding}~\cite[proof of Thorem 2, Equation (3.8)]{lu2001spatially}, which has a similar form as~\eqref{eq:similarform}, and easily provides a short-time bound that satisfies the above mentioned properties.
\end{nb}

\begin{nb} As explained in the Introduction of~\cite{liulu}, the above properties (existence, uniqueness, moments estimates and entropy dissipation) of solutions to \eqref{eq:BFDequation}, obtained for collision kernel of the form \eqref{eq:HypBB} are easily extended to more general -- and physically relevant -- interaction kernels of the form described in Section~\ref{app:COKer}. We refer to~\cite{liulu} as well as~\cite{HeLuPu} for more details about this question.\end{nb}

\subsection{Csisz\'ar-Kullback-Pinsker inequalities} 

We here state a recent improvement of the usual Csisz\'ar-Kullback-Pinsker inequality (CKP inequality). We only recall here that the usual, original CKP inequality asserts that
$$\left\|f-\M_0^f\right\|_1^2 \leq 2\|f\|_1\,\mathcal{H}_0\left(f\big|\M_0^f\right)$$
for any $f \in L^1_2(\R^3)$ where $\M_0^f$ is the Maxwellian state with same mass, momentum and energy as $f$. 

Such a result has been extended recently by the first author and generalized to Fermi-Dirac relative entropy yielding the following weighted $L^p$-version of the CKP inequality
\begin{prop}\label{prop:CKP} For any $\e > 0$ and $0 \leq f \in L^{1}_2(\R^3) \setminus \{0\}$ satisfying~\eqref{eq:fEf} and such that $1-\e f \geq0$ and $r_E > \frac{2}{5}$, we have 
$$\|f-\M_\e^f\|_{L^1_2}^2 \leq 8\|\M_\e\|_{L^1_4} \; \mathcal{H}_\e\left(f\big|\M_\e\right)$$
and, for any weight function $\varpi\::\:\R^3\to \R_+$ and any $p \in [1,2]$, it holds
$$\left\|\varpi \left(f-\M_\e^f\right)\right\|_p^2 \leq 2\max\left(\left\|\varpi^2\M_\e^f\right\|_{\frac{p}{2-p}}\,,\,\|\varpi^2\,f\|_{\frac{p}{2-p}}\right)\mathcal{H}_\e\left(f\big|\M_\e^f\right),$$
where $\M_\e^{f}$ is the Fermi-Dirac distribution associated to $f$.
\end{prop}
\begin{nb} Recall that, as observed in the Introduction, the assumption $r_E >\frac{2}{5}$ is equivalent to $\e < \e_{\mathrm{sat}}$ where $\e_{\mathrm{sat}}$ is defined in~\eqref{eq:sat} and this implies that $f \neq F_\e$ where $F_\e$ is the saturated steady state defined in~\eqref{eq:dege}, and the existence of $\M_\e^f$. Of course, in the above result, for $p=2$, we adopt the convention $\|\cdot\|_{\frac{p}{2-p}}=\|\cdot\|_{\infty}.$\end{nb}

\section{\texorpdfstring{Explicit and uniform-in-$\e$ bounds for Fermi-Dirac statistics}{annex Fermi-Dirac statistics}} \label{appendix:FDS}

In this last Appendix section, we provide, in the following Lemma~\ref{lemma:appendixFDS}, explicit uniform-in-$\e$ bounds on the $L^{\infty}$ and $L^1_k$, $k \geq 0$, norms of the Fermi-Dirac Statistics.

\begin{lem} \label{lemma:appendixFDS}
Consider a nonnegative $f \in L^1_2(\R^3)$, which density, average velocity and temperature we denote respectively by $\varrho$, $u$ and $E$. We let
$$
\e^{\dag}_{\rm sat} := 2^{\frac52} \cdot 3^{\frac32} \cdot 5^{-\frac52} \cdot \pi^{\frac32} \, \varrho^{-1} \, E^{\frac32}.
$$
Then there exist $\mathbf{C}_{\M,\infty}>0$ and,  $\mathbf{C}_{\M,1,k}>0$, for any $k\geq 0$, which are explicit and depend only on $\varrho$, $u$ and $E$, such that for any $\e \in (0,\e^{\dag}_{\rm sat}]$, we have
\begin{equation} \label{eqlemmaapendixLinfty}
\|\M^f_\e\|_{\infty} \leq \mathbf{C}_{\M,\infty},
\end{equation}
and, for any $k\geq 0$,
\begin{equation} \label{eqlemmaapendixL1k}
\|\M^f_\e\|_{L^1_k} \leq \mathbf{C}_{\M,1,k}.
\end{equation}
\end{lem}

\smallskip

\noindent Notice in particular that $\e^{\dag}_{\rm sat} = 2^{\frac12} \cdot 3^{\frac52} \cdot 5^{-4} \cdot \pi^{\frac12} \, \e_{\rm sat} \sim 0.06 \; \e_{\rm sat}$.

\begin{proof}
For the sake of simplicity, we prove the lemma with the extra assumption $u=0$ and briefly discuss the case $u \neq 0$ at the end of the proof. Let us recall the notation, for any $v \in \R^3$,
$$
\M_{\e}^{f} (v) = \frac{e^{a_{\e} + b_{\e} |v|^2}}{1 + \e e^{a_{\e} + b_{\e} |v|^2}},
$$
as well as
$$
 r_E:=\frac{E}{\bm{T}(\varrho,\e)}, \qquad \qquad \bm{T}(\varrho,\e):=\frac{1}{2}\left(\frac{3\e\varrho}{4\pi}\right)^{\frac{2}{3}}.
$$
It is proven in~\cite[Proposition 18]{borsoni2023extending} that for any $\e \in (0,\e_{\rm sat})$, as soon as
\begin{equation} \label{eq:conditionfinalproof}
r_E \geq r^{\dag}, \qquad \qquad r^{\dag} := \left( \frac{4}{\pi}\right)^{\frac13} \left( \frac53\right)^{\frac53},
\end{equation}
we have
$$
\e \, e^{a_{\e}} \leq \frac23 \left( \frac{r_E}{r^{\dag}}\right)^{-\frac32}.
$$
The above can be reformulated as
$$
e^{a_{\e}} \leq \frac23 \, {r^{\dag}}^{\frac32} \; {E}^{-\frac32} \, \left( \frac{3 \varrho}{4\pi}\right),
$$
which, by direct computation, gives in particular
\begin{equation} 
e^{a_{\e}} \leq \varrho \, {E}^{-\frac32}.
\end{equation}
We notice that, letting
\begin{equation}
    \e^{\dag}_{\rm sat} := 2^{\frac52} \cdot 3^{\frac32} \cdot 5^{-\frac52} \cdot \pi^{\frac32} \, \varrho^{-1} \, E^{\frac32},
\end{equation}
the condition~\eqref{eq:conditionfinalproof} is equivalent to assuming $\e \in (0,\e^{\dag}_{\rm sat}]$. In this event, we have proven~\eqref{eqlemmaapendixLinfty} with $\mathbf{C}_{\M,\infty} = \varrho \, {E}^{-\frac32}$, since $$\|\M_{\e}^{f}\|_{\infty} \leq e^{a_{\e}}.$$
Moving on to the proof of~\eqref{eqlemmaapendixL1k}, and still assuming $\e \in (0,\e^{\dag}_{\rm sat}]$, we have by definition, for any $k \geq 0$,
$$
\|\M^f_\e\|_{L^1_k} = \int_{\R^3} \M^f_\e(v) \, \langle v \rangle^k \, \d v,
$$
so that
\begin{equation}
\|\M^f_\e\|_{L^1_k} \leq e^{a_{\e}} \int_{\R^3} e^{b_{\e}|v|^2}\, \langle v \rangle^k \, \d v.
\end{equation}
Passing to spherical coordinates $v \in \R^3 \setminus\{0\} \mapsto (r,\sigma) \in \R^*_+ \times \Sb^2$ and then performing the change of variables $r \in \R_+^* \mapsto x = \sqrt{-b_{\e}} \, r \in \R_+^*$ (recall that $b_{\e}$ is negative), we get
\begin{equation} \label{eqprooflemmaappendixL1kintermediate}
\|\M^f_\e\|_{L^1_k} \leq e^{a_{\e}} \, |\Sb^2| \, \left(|b_{\e}|^{\frac32} \min (1, \, |b_{\e}|)^k \right)^{-1} \int_{\R_+} e^{-x^2}\, x^2(1 + x^2)^{\frac{k}{2}} \, \d x,
\end{equation}
where we used the fact that $1 + \left(\frac{x}{\sqrt{-b_{\e}}} \right)^2 \leq \min(1, |b_{\e}|)^{-1} \, (1 + x^2)$. Letting $$\mathbf{C}^0_{\M,1,k} = \varrho \, {E}^{-\frac32} \, 4 \, \pi \int_{\R_+} e^{-x^2}\, x^2(1 + x^2)^{\frac{k}{2}} \, \d x,$$
which is explicit and has the same properties as $\mathbf{C}_{\M,\infty}$, Equation~\eqref{eqprooflemmaappendixL1kintermediate} yields, as we just proved $e^{a_\e} \leq \varrho \, {E}^{-\frac32}$,
\begin{equation} \label{eqprooflemmaappendixL1kintermediate2}
\|\M^f_\e\|_{L^1_k} \leq \mathbf{C}^0_{\M,1,k} \, \left(|b_{\e}|^{\frac32} \min (1, \, |b_{\e}|)^k \right)^{-1}.
\end{equation}
Let us now provide an explicit lower bound for $|b_{\e}|$. We follow a similar reasoning to the one of~\cite[proof of Proposition 18]{borsoni2023extending}  It is proven in~\cite[proof of Proposition 3]{lu2001spatially} that, letting
$$
I_s(\tau) := \int_{0}^{\infty} \frac{r^s}{1 + \tau \, e^{r^2}} \, \d r, \qquad P(\tau) := I_4(\tau) \, I_2(\tau)^{-5/3}, \qquad \tau \in \R_+^*,
$$
the function $P$ is increasing from $\R_+^*$ to $\left(\frac{3^{5/3}}{5}, +\infty \right)$, and we have
\begin{equation} \label{eqPannexalmostfinished}
\left(\frac{\e}{4 \pi} \right)^{2/3} P \left( \frac{1}{\e e^{a_{\e}}} \right) = \frac{3 \varrho \, E}{{\varrho}^{5/3}}, \qquad b_{\e} = -\left( \frac{4 \pi }{\e \varrho} \, I_2\left( \frac{1}{\e e^{a_{\e}}} \right) \right)^{\frac23}.
\end{equation}
As by definition of $P$, $I_2(\tau) = \left(\frac{I_4(\tau)}{P(\tau)} \right)^{\frac35}$ for any $\tau \geq 0$, we deduce that, using both equations in~\eqref{eqPannexalmostfinished},
$$
|b_{\e}| = \left( \frac{4 \pi }{\e \varrho} \left\{ I_4\left( \frac{1}{\e e^{a_{\e}}} \right) \, \left(\frac{\e}{4 \pi} \right)^{2/3} \frac{{\varrho}^{5/3}}{3 \varrho \, E} \right\}^{\frac35}\right)^{\frac23},
$$
which, raised to the power $\frac52$, rewrites
\begin{equation} \label{eqannexbeps52}
    |b_{\e}|^{\frac52} =  \left(\frac{4 \pi }{3 \, \e \, \varrho \, E} \right) I_4\left( \frac{1}{\e e^{a_{\e}}} \right).
\end{equation}
As it holds for any $r \geq 0$ and $\tau \geq 0$ that
$$
\frac{1}{1 + \tau \, e^{r^2}} \geq \frac{e^{-r^2}}{1 + \tau},
$$
we have for any $\tau \geq 0$ that
$$
I_4(\tau) \geq  \frac{1}{1 + \tau}\int_{0}^{\infty} r^4 \, e^{-r^2} \, \d r = \frac{1}{1 + \tau} \cdot \frac{3}{8 \sqrt{\pi}},
$$
where the constant on the right-hand-side was computed from $\frac12\Gamma \left(\frac{5}{2} \right)$, with $\Gamma$ standing (here only) for the Gamma function. It thus comes that
$$
I_4\left( \frac{1}{\e e^{a_{\e}}} \right) \geq \frac{3}{8 \sqrt{\pi}} \, \frac{\e e^{a_{\e}}}{1 + \e e^{a_{\e}}}.
$$
Since $\e \leq \e^{\dag}_{\rm sat}$, we can again make use of the result of~\cite[Proposition 18]{borsoni2023extending} (as, with their notation, $\gamma/\gamma^{\dag} \geq 1$) to obtain that
$$
\e e^{a_{\e}} \leq \frac23,
$$
implying $1 + \e e^{a_{\e}} \leq 1 + \frac23 = \frac53$, so that, also recalling~\eqref{eqannexbeps52}, we have
\begin{equation}
|b_{\e}|^{\frac52} \geq \left(\frac{4 \pi }{3\, \varrho \, E} \right) \frac{9}{8 \cdot 5 \sqrt{\pi}} \; \; e^{a_{\e}}.
\end{equation}
The last part of the proof consists in proving a lower bound for $e^{a_\e}$. We use a similar reasoning as previously. As it holds for any $r \geq 0$ and $\tau \geq 0$ that
$$
\frac{e^{-r^2}}{\tau}\geq \frac{1}{1 + \tau \, e^{r^2}} \geq \frac{e^{-r^2}}{1 + \tau},
$$
we deduce that, for any $\tau \geq 0$,
$$
P(\tau) \geq \frac{\tau^{\frac53}}{1 + \tau} \, \cdot \, \frac{3 \cdot 2^{\frac13}}{\pi^{\frac13}},
$$
where the constant on the right-hand-side comes from $\frac12 \Gamma \left(\frac{5}{2} \right)  \left( \frac12 \Gamma \left(\frac{3}{2} \right) \right)^{-\frac53}$. Setting $\tau = (\e e^{a_\e})^{-1}$ in the previous inequality, and using~\eqref{eqPannexalmostfinished}, we obtain
$$
\frac{3 \varrho \, E}{{\varrho}^{5/3}} \left(\frac{4 \pi}{\e} \right)^{2/3} \geq \frac{(\e e^{a_{\e}})^{-\frac23}}{1 + \e e^{a_{\e}}}\, \cdot \, \frac{3 \cdot 2^{\frac13}}{\pi^{\frac13}}.
$$
Again, since $1 + \e e^{a_{\e}} \leq \frac53$, we thus have obtained, for any $\e \in (0,\e^{\dag}_{\rm sat}]$,
$$
e^{a_{\e}} \geq \left(\frac{3}{10 \pi} \right)^{\frac32} \varrho \, E^{-\frac32}.
$$
Defining finally
$$
\mathbf{C}_{\M,1,k} = \mathbf{C}^0_{\M,1,k} \, \left(\mathbf{b}^{\frac32} \min (1, \, \mathbf{b})^k \right)^{-1}, \qquad \quad \mathbf{b} = \frac{3}{10 \, \pi^{\frac25}} \;  E^{-1}\,,
$$
this  concludes the proof of ~\eqref{eqlemmaapendixL1k}.
\end{proof}

\bibliographystyle{plain}
\bibliography{biblio}

\end{document}